\documentclass[a4paper,11pt]{article}
\usepackage{amsmath}
\usepackage{amsthm}
\usepackage{amssymb}
\usepackage[latin1]{inputenc}
\usepackage[english]{babel}
\usepackage{hyperref}
\usepackage{color}

\theoremstyle{plain}
\newtheorem{thm}{Theorem}
\newtheorem{prop}[thm]{Proposition}
\newtheorem{lem}[thm]{Lemma}
\newtheorem{cor}[thm]{Corollary}

\theoremstyle{definition}
\newtheorem{defn}[thm]{Definition}
\theoremstyle{remark}

\DeclareMathOperator{\End}{End}
\DeclareMathOperator{\tr}{tr}
\DeclareMathOperator{\scal}{scal}
\DeclareMathOperator{\Diff}{Diff}
\DeclareMathOperator{\id}{id}
\DeclareMathOperator{\ad}{ad}
\DeclareMathOperator{\Ad}{Ad}

\begin{document}
\title{Special homogeneous almost complex structures on symplectic manifolds}
\date{\today}

\author{Alberto Della Vedova\footnote{Università di Milano-Bicocca (Italy) - alberto.dellavedova@unimib.it}
%\footnote{Dipartimento di Matematica e Applicazioni - Università di Milano-Bicocca. Via R. Cozzi, 55 - 20125 Milano - Italy. E-mail: alberto.dellavedova@unimib.it}
}

\maketitle

\begin{abstract}
Homogeneous compatible almost complex structures on symplectic manifolds are studied, focusing on those which are special, meaning that their Chern-Ricci form is a multiple of the symplectic form. 
Non Chern-Ricci flat ones are proven to be covered by co-adjoint orbits.
Conversely, compact isotropy co-adjoint orbits of semi-simple Lie groups are shown to admit special compatible almost complex structures whenever they satisfy a necessary topological condition.
Some classes of examples including twistor spaces of hyperbolic manifolds and discrete quotients of Griffiths period domains of weight two are discussed.
\end{abstract}

%\tableofcontents

\section{Introduction}
The aim of the present paper is provide some new examples of symplectic manifolds endowed with compatible almost complex structures having special curvature properties.
More specifically, on a symplectic manifold $(M,\omega)$ we will call \emph{special} a compatible almost complex structure $J$ whose Chern-Ricci form $\rho$ is a multiple of the symplectic form $\omega$ (cfr. definition \ref{defn::specialJ}).
This condition has been previously considered by Apostolov and Dr\u aghici \cite[Section 7]{ApoDra2003}, and should be thought of as a natural extension of the K\"ahler-Einstein condition to non-integrable compatible almost complex structures.
As in the integrable case, the existence of a special $J$ forces the first Chern class $c_1$ of $(M,\omega)$ to be a multiple of $[\omega]$.
Since $\rho$ represents $4\pi c_1$, the existence of a special $J$ satisfying $\rho=\lambda \omega$ makes $(M,\omega)$ a symplectic Fano, symplectic Calabi-Yau or symplectic general type depending on $\lambda$ is positive, zero or negative respectively.
On the other hand, it is a remarkable fact that on closed symplectic Calabi-Yau or symplectic Fano manifolds the Riemannian metric associated to a special $J$ is Einstein if and only if $J$ is integrable (cfr. proposition \ref{prop::nonintnonEinst}).

Similarly to the K\"ahler-Einstein case, the existence of a special compatible almost complex structure may put some topological constraints on the symplectic manifold \cite{Yau1977}.
On the other hand, the existence problem for special compatible almost complex structures may be included in a moment map picture similar to the one discussed by Donaldson for K\"ahler-Einstein metrics \cite{Donaldson2015}.
We leave these considerations for future work, instead here we focus on discussion of examples.

The first large class is constituted by all closed symplectic manifolds $(M,\omega)$ satisfying the necessary condition $c_1=\lambda [\omega]$ and admitting an integrable, and K-stable if $\lambda>0$, compatible almost complex structure $J_0$.
Standing on fundamental works of Aubin, Yau, Chen-Donaldson-Sun and Tian on the existence of K\"ahler-Einstein metrics on closed complex manifolds \cite{Aubin1978,Yau1978,ChenDonaldsonSun2015,Tian2015}, it is easy to see that this is the class of closed K\"ahler-Einstein manifolds (cfr. proposition \ref{prop::existspecialifint} and discussion thereafter).
Starting with the integrable special $J$ which exists on any symplectic manifold of this class, one can try and perturb $J$ to a non-integrable special compatible almost complex structure, as done by Lejmi for locally toric K\"ahler-Einstein surfaces \cite[Corollary 4.3]{Lejmi2012}.
On the other hand, no general existence results are avilable for special compatible complex structures on symplectic manifolds which are not of K\"ahler type, meaning that they admit no integrable compatible almost complex structures.
This lead us to focus primarily on this class. 

Our source of special compatible almost complex structures over non-K\"ahler type symplectic manifolds will be symplectic manifolds endowed by homogeneous compatible almost complex structures.
More specifically, we consider a symplectic manifold $(M,\omega)$ and a compatible almost complex structure $J$ on it and we assume there is a transitive action of a connected Lie group $G$ which preserves both $\omega$ and $J$.
In these circumstances, $\omega$ and $J$ are completely determined by their value at a single point $x \in M$.
Denoted by $K$ the isotropy subgroup at $x$, clearly $M$ turns out to be diffeomorphic to the homogeneous space $G/K$.
For simplicity we also assume that the action of $G$ is almost-effective, meaning that no non-discrete normal subgroups of $G$ act trivially on $M$.
As a consequence, we can work as if $K$ were compact (cfr. subsection \ref{subsec::gensetup}).
The upshot is that $\omega$ and $J$ are determined by a linear two-form $\sigma$ and a linear endomorphism $H$ on the Lie algebra $\mathfrak g$ of $G$ satisfying suitable conditions (cfr. subsection \ref{subsec::gensetup} and discussion after theorem \ref{thm::chernriccionm}).
Due to $G$-invariance, any geometric quantity related to $\omega$ and $J$ is determined by its value at $x$.
As a consequence, one may expect that this value at $x$ can be expressed in terms of $\sigma$ and $H$.
Indeed, we will show that the Chern-Ricci form $\rho$ of $J$ satisfies (cfr. corollary \ref{cor::rhohomwithW})
\begin{equation}\label{eq::intro_rhohom}
\rho(X,Y)_x = \tr \left( \ad_{H[X,Y]_{\mathfrak g}} - H \ad_{[X,Y]_{\mathfrak g}} \right).
\end{equation}
Similarly, the Hermitian scalar curvature of $J$ is the real constant given by (cfr. corollary \ref{cor::Hermitianscalarcurvhom})
\begin{equation}
s = \tr \left( \ad_{H\xi} - H \ad_\xi \right),
\end{equation}
where $\xi \in \mathfrak g$ is an element depending just on $\sigma$ (cfr. proposition \ref{prop::propertiesxi}).
Note that these results reduce to well known formulae for the Ricci form and the scalar curvature of compact homogeneous K\"ahler manifolds when $J$ is integrable \cite[Proposition 8.47 and Corollary 8.49]{Besse1987}.
Finally also an explicit formula for the squared norm of the Nijenhuis tensor of $J$ can be deduced (cfr. proposition \ref{prop::formulaNijenhuis}).

When $G$ is compact and semi-simple, by classical results of Borel and Weil it follows that $(M,\omega)$ is of K\"ahler type \cite{Serre1954bis}.
Therefore, in view of getting non-K\"ahler symplectic manifolds, one has to consider primarily the case when $G$, hence $M$, is non-compact.
Eventually, under unimodularity assumption of $G$, compact manifolds are obtained by taking the quotient $\Gamma \backslash M$, being $\Gamma \subset G$ a co-compact discrete subgroup (cfr. subsection \ref{subsection::quot}).

The simplest situation is when $G$ is a symplectic Lie group, i.e. a Lie group endowed with a left-invariant symplectic structure, and $K$ is the trivial subgroup.
As noted by Chu and by Lichnerowicz and Medina, if in addition $G$ is assumed to be unimodular, then it has to be solvable \cite{Chu1974,LichnerowiczMedina1988}.
Therefore any compact symplectic manifold of the form $\Gamma \backslash G$ is a solvmanifold.
While the existence of a special compatible almost complex structure on a symplectic solvmanifold may be obstructed, some restricted subclasses of symplectic solvmanifold do admit plenty of such structures.   
In particular, equation \eqref{eq::intro_rhohom} readily implies that any homogeneous compatible almost-complex structure on a two-step nilpotent Lie group is Chern-Ricci flat, i.e. it satisfies $\rho=0$.
As a consequence, by taking quotients we recover the following result which was already known by a work of Vezzoni \cite{Vezzoni2013}.
\begin{thm}
Any homogeneous compatible almost complex structure on a symplectic two-step nilmanifold is Chern-Ricci flat.
\end{thm}
Thus, in particular, any homogeneous almost complex structure on the Kodaira-Thurston manifold \cite{Thurston1976} is Chern-Ricci flat (cfr. subsection \ref{subsec::sympliegroups}).

A richer class of homogeneous symplectic manifolds which drawn considerable attention by many point of view is that of co-adjoint orbits endowed with their canonical Kirillov-Kostant-Souriau symplectic structure \cite{Vogan2000}.
In our situation co-adjoint orbits serve as covering spaces of symplectic manifolds admitting homogeneous non Chern-Ricci flat special compatible almost complex structures.
In fact, it holds the following (cfr. subsection \ref{subsect::coadjointorbits})
\begin{thm}
Let $(M,\omega)$ be a symplectic manifold admitting a homogeneous compatible almost complex structure satisfying $\rho = \lambda \omega$ for some $\lambda \neq 0$.
Then $M$ is a covering space of a co-adjoint orbit and $\omega$ is the pull-back via the covering map of the canonical symplectic form.
\end{thm}
A partial converse of theorem above is represented by the following
\begin{thm}\label{thm::intro_specialonsemisimplecoadjoint}
Let $G$ be a connected semi-simple Lie group, and let $M \subset \mathfrak g^*$ be a co-adjoint orbit equipped with the canonical symplectic form $\omega$.
Assume that the isotropy of $M$ is compact and contains no non-discrete normal subgroups of $G$.
If the first Chern class of $(M,\omega)$ satisfies $4\pi c_1=\lambda [\omega]$ for some $\lambda \in \mathbf R$, then there exists a homogeneous special almost complex structure on $(M,\omega)$.
\end{thm}
The proof of this result is completely constructive.
As an application, taking $G=SO(2n,1)$ we exhibit special compatible almost complex structures on the twistor space of the hyperbolic space $H^{2n}$ (cfr. subsections \ref{subsubsect::twistorHspace}).
The quotient of this twistor space by a discrete torsion-free subgroup $\Gamma \subset SO(2n,1)$ is the twistor space of the hyperbolic manifold $\Gamma \backslash H^{2n}$, which has shown to be symplectic by Reznikov \cite{Reznikov1993}.
More recently, Fine and Panov showed that it is symplectic general type, symplectic Calabi-Yau, or symplectic Fano if $n=1$, $n=2$ or $n>2$ respectively \cite{FinePanov2009}.
We give an independent proof of these results and we complement them by the following (cfr. subsection \ref{subsec::quotienttwistorsandperioddomains})

\begin{thm}
The twistor space $(M,\omega)$ of a hyperbolic $2n$-fold admits a compatible almost complex structure satisfying $\rho=(2n-4) \omega$. 
\end{thm}
The fact that $\omega$ coincides up to a factor with the Reznikov symplectic form can be checked by direct calculations.
On the other hand, it is straightforward to recover the Fine and Panov result by taking cohomology classes of $\rho$ and $[\omega]$.
All compact examples above turn out be of non-K\"ahler type precisely when $n>1$.
Therefore none of them is both non-K\"ahler type and symplectic general type.
In order to get such kind of manifolds, one can apply the construction behind the theorem \ref{thm::intro_specialonsemisimplecoadjoint} to some coadjoint orbits of the group $G=SO(2p,q)$.
The resulting manifolds are Griffiths period domains of weight two (cfr. subsection \ref{subsubsection::gentypesymp}).

\begin{thm}
A Griffiths period domain of weight two $SO(2p,q)/U(p) \times SO(q)$ admits a symplectic form $\omega$ and a compatible almost complex structure satisfying $\rho=(2p-2q-2) \omega$. 
\end{thm}

In this situation, the symplectic form and the special almost complex structure descend to the quotient by any discrete torsion-free subgroup $\Gamma \subset G$ and give rise to a symplectic general type, symplectic Calabi-Yau, or symplectic Fano manifold depending on the sign of $p-q-1$.
Compact quotients exist \cite{Borel1963}, and they are of non-K\"ahler type whenever $p>1$ and $q \neq 2$ (cfr. subsection \ref{subsec::quotienttwistorsandperioddomains}).

\bigskip

The paper is organized as follows. 
Section \ref{sec::specialcacs} is devoted to recall some well known facts about compatible almost complex structures on symplectic manifolds and to introduce special compatible almost complex structures.
In section \ref{sec::homsympman} is developed some theory of homogeneous almost complex structures on symplectic manifold. 
Finally, in section \ref{sec::applications} this theory is applied to some symplectic Lie group and co-adjoint orbits, and (compact) quotients of these manifolds are discussed.

\section{Special compatible almost complex structures}\label{sec::specialcacs}

In order to fix notation, here we recall some well known basic facts about symplectic manifolds and compatible almost complex structures on them.
At the end of this section we will define compatible almost complex structures having special curvature properties. 
These structures will constitute our main topic.

Let $(M,\omega)$ be a symplectic manifold of dimension $2n$ endowed with a compatible almost complex structure $J$.
By definition, $M$ is a differentiable $2n$-fold equipped with a closed two-form $\omega$ such that $\omega^n$ is nowhere vanishing on $M$.
Moreover, $J$ is a smooth section of the vector bundle $\End(TM)$ such that $-J^2$ is the identity endmorphism of $TM$, and for any pair of non-zero vectors $X,Y$ which are tangent at the same point of $M$ it holds
\begin{equation}
\omega(JX,JY) = \omega(X,Y), \qquad \omega(X,JX)>0.
\end{equation}
To any such $J$, is then associated a Riemannian metric $g$ defined by
\begin{equation*}
g(X,Y) = \omega(X,JY).
\end{equation*}
Sometimes, the compatible almost complex structure $J$ may come from a holomorphic atlas on $M$.
A celebrated theorem of Newlander and Nirenberg \cite{NewlanderNirenberg1953} states that the only obstruction for that is constituted by the Nijenhuis tensor $N$ of $J$, which is a $TM$-valued two-form on $M$ defined by
\begin{equation}
4 N(X,Y)= [JX,JY]-J[JX,Y]-J[X,JY]-[X,Y].
\end{equation}
If $N$ is identically zero, then $J$ is said to be integrable.
Therefore, when $J$ is integrable we are in the familiar situation of K\"ahler manifolds.
Under these circumstances, $(M,J)$ is a complex manifold by Newlander-Nirenberg theorem, and $g$ turns out to be a K\"ahler metric on it with K\"ahler form $\omega$.
In this paper we are primarily interested in the non-integrable case, when the metric $g$ is said to be almost-K\"ahler.

To any compatible almost-complex structure $J$ as above, beside the Levi-Civita connection $D$ of the metric $g$, is associated the Chern connection $\nabla$, that is the unique metric connection on $(M,g)$ whose torsion is equal to the Nijenhuis tensor $N$.
Clearly $\nabla$ coincides with $D$ if and only if $J$ is integrable.
In general, $D$ and $\nabla$ are related by the identity
\begin{equation}\label{eq::nablaintermsofD}
\nabla_X = D_X - \frac{1}{2}J(D_XJ),
\end{equation}
as can be easily checked by the Koszul formula.
The Chern connection is invariant under the action of symplectic biholomorphisms, that is diffeomorphisms of $M$ that preserve both $\omega$ and $J$.
As a consequence one has the following result, which we highlight here for future reference.
\begin{lem}\label{lem::LiederCherncon}
If $X$ is a symplectic and holomorphic vector field on $M$ then
\begin{equation}
[X,\nabla_YZ] = \nabla_{[X,Y]}Z + \nabla_Y[X,Z]
\end{equation}
holds for all vector fields $Y,Z$ on $M$.
\end{lem}
\begin{proof}
Since $X$ is symplectic, by definition one has $L_X\omega=0$.
Hence, for all vector field $W$ on $M$ one has
\begin{equation}\label{eq::omegaLieXnabla}
\omega([X,\nabla_YZ],W) = X\left(\omega(\nabla_YZ,W)\right)
%- (L_X\omega)(\nabla_YZ,W)
- \omega(\nabla_YZ,[X,W]).
\end{equation}
Therefore a straightforward calculation involving Koszul formula yields
\begin{multline}
X\left(\omega(\nabla_YZ,W)\right)
= \omega(\nabla_YZ,[X,W]) 
+ \omega(\nabla_{[X,Y]}Z,W)
+ \omega(\nabla_Y[X,Z],W) \\
+ \omega\left([X,N(W,Z)],Y\right)
- \omega\left(N(W,[X,Z]),Y\right)
- \omega\left(N([X,W],Z),Y\right).
\end{multline}
Since $X$ is holomorphic, $L_XJ=0$, hence $L_XN=0$.
As a consequence the second line vanishes.
Therefore the thesis follows by the arbitrariness of $W$ after substituting in \eqref{eq::omegaLieXnabla}.
%\begin{eqnarray*}
%&& 4 [X,N(W,Z)] -4 N(W,[X,Z]) -4 N([X,W],Z) \\
%&=& [X,[JW,JZ]] - J[X,[JW,Z]] - J[X,[W,JZ]] - [X,[W,Z]] \\
%&& - [JW,J[X,Z]] + J[JW,[X,Z]] + J[W,J[X,Z]] + [W,[X,Z]] \\
%&& - [J[X,W],JZ] + J[J[X,W],Z] + J[[X,W],JZ] + [[X,W],Z] \\
%&=& [X,[JW,JZ]] - J[X,[JW,Z]] - J[X,[W,JZ]] - [X,[W,Z]] \\
%&& - [JW,[X,JZ]] + J[JW,[X,Z]] + J[W,[X,JZ]] + [W,[X,Z]] \\
%&& - [[X,JW],JZ] + J[[X,JW],Z] + J[[X,W],JZ] + [[X,W],Z] \\
%&=& [X,[JW,JZ]] - J[X,[JW,Z]] - J[X,[W,JZ]] - [X,[W,Z]] \\
%&& + [JW,[JZ,X]] - J[JW,[Z,X]] - J[W,[JZ,X]] - [W,[Z,X]] \\
%&& + [JZ,[X,JW]] - J[Z,[X,JW]] - J[JZ,[X,W]] - [Z,[X,W]] \\
%&=& 0.
%\end{eqnarray*}
\end{proof}

A consequence of \eqref{eq::nablaintermsofD} above is that both $J$ and $\omega$ are parallel with respect to $\nabla$.
In particular, $\nabla$ is a complex-linear connection on the complex vector bundle $(TM,J)$.
Thus Chern-Weil theory apply, and the Chern classes of $(TM,J)$ can be expressed in terms of the curvature $R \in \Omega^2(\End(TM))$ of $\nabla$, which is defined by
\begin{equation}
R(X,Y) = \nabla_X \nabla_Y - \nabla_Y \nabla_X - \nabla_{[X,Y]}.
\end{equation}
As it is well known, the space of all compatible almost complex structures on $(M,\omega)$ is contractible, for it can be realized as a space of sections of a fiber bundle over $M$ with contractible fiber. %(see, for example, the book of McDuff and Salamon on this elementary fact \cite{McDoffSalamon1998})
As a consequence, the Chern classes $c_k \in H_{dR}^{2k}(M)$ of $(TM,J)$ in fact do not depend on the compatible almost-complex structure $J$ chosen.
For this reason $c_k$ is called the $k$-th Chern class of $(M,\omega)$.
In this paper we will deal just with the first Chern class $c_1 \in H_{dR}^2(M)$.

The cohomology class $4\pi c_1$ can be represented by the Chern-Ricci form of $J$, that is the two form $\rho$ on $M$ defined by
\begin{equation*}
\rho(X,Y) = \tr(JR(X,Y)),
\end{equation*}
where $\tr$ denotes the real trace.
When $J$ is integrable, $\rho$ is nothing but (twice) the Ricci form of the K\"ahler metric $g$.

Attached to $J$ there is also a smooth function $s$ on $M$ called the Hermitian scalar curvature of $J$, which is defined by the identity
\begin{equation}
s \omega = n \rho \wedge \omega^{n-1}.
\end{equation}
In case $M$ is closed, i.e. compact with no boundary, it follows readily integrating over $M$ that the average of $s$ depends just on the cohomology classes $[\omega],c_1 \in H_{dR}^2(M)$.

As one may expect, by \eqref{eq::nablaintermsofD} one can express $R$, $\rho$ and $s$ in terms of more familiar Riemannian quantities related to the metric $g$ and $DJ$ or $N$.
The upshot for the Hermitian scalar curvature is the identity
\begin{equation}\label{eq::s=scal+2snormN}
s = \scal + 2|N|^2,
\end{equation}
where $\scal$ denotes the scalar curvature of the Riemannian metric $g$, and $|N|$ is the norm of the tensor $N$ with respect to $g$. 

Our main interest here is in compatible almost-complex structures $J$ having a special curvature property which we now introduce. This kind of curvature properties was previously considered by Apostolov and Dragichi \cite[Section 7]{ApoDra2003}.
\begin{defn}\label{defn::specialJ}
A compatible almost complex structure $J$ on $(M,\omega)$ is called \emph{special} if there exists $\lambda \in \mathbf R$ such that
the Chern-Ricci form of $J$ satisfies
\begin{equation}\label{eq::specialequation}
\rho = \lambda \omega.
\end{equation}
\end{defn}
A number of comments to this definition are in order.
It is clear that the existence of a special compatible almost complex structure poses a cohomological constraint on $(M,\omega)$.
Indeed, if such a structure exists, then by discussion above $c_1, [\omega] \in H_{dR}^2(M)$ must satisfy
\begin{equation}
4 \pi c_1 = \lambda [\omega].
\end{equation}
In particular, special compatible almost complex structures with $\lambda=0$ or $\lambda>0$ may exist just on so-called symplectic Calabi-Yau  or symplectic Fano manifolds respectively. 

If $J$ is special, then it has constant Hermitian scalar curvature equal to $n\lambda$.
Nonetheless, it follows by \eqref{eq::s=scal+2snormN} that the scalar curvature of the associated Riemannian metric $g$ is non-constant unless $|N|$ is constant as well.

If $J$ is special and integrable, then $g$ is K\"ahler-Einstein. Therefore equation \eqref{eq::specialequation} generalizes the K\"ahler-Einstein equation.
However it should be noted that the Riemannian metric $g$ associated with a special compatible almost complex structure is not Einstein in general.
For example, this happens whenever the norm of the Nijenhuis tensor $|N|$ is non-constant, as can be readily checked by \eqref{eq::s=scal+2snormN}. 
Furthermore, under the additional assumption that the symplectic manifold is closed, one may expect that $g$ is Einstein if and only if $J$ is integrable.
This might be true even dropping the assumption that $J$ is special.
Such more general expectation goes under the name of Goldberg conjecture, which states that, on a closed symplectic manifolds, a compatible almost complex structure $J$ having Einstein associated Riemannian metric $g$ is integrable \cite{Goldberg1969}.
While the general case is open at the present time, this conjecture has been proved by Sekigawa assuming the scalar curvature of $g$ non-negative \cite{Sekigawa1987}.
Later, the assumption on the scalar curvature has been relaxed a bit by Dr\u aghici, who proved the conjecture assuming the inequality $c_1 \cup [\omega]^{n-1} \geq 0$ instead of the non-negativeness of the scalar curvature \cite[Theorem 1]{Draghici2005}.
As a consequence one has the following integrability criterion for special almost complex structures on symplectic Calabi-Yau and symplectic Fano manifolds

\begin{prop}\label{prop::nonintnonEinst}
Let $J$ be a compatible almost complex structure on a closed symplectic manifold satisfying $\rho = \lambda \omega$ for some $\lambda \geq 0$.
If the associated Riemannian metric is Einstein, then $J$ is integrable. 
\end{prop}

On the other hand, at least when $\lambda \leq 0$, any closed symplectic manifold admitting an integrable compatible almost complex structure has also a special one, in fact integrable and hence K\"ahler-Einstein.
More precisely we have the following
\begin{prop}\label{prop::existspecialifint}
Let $(M,\omega)$ be a closed symplectic manifold such that $4\pi c_1 = \lambda [\omega]$ for some real $\lambda \leq 0$, and assume there exists an integrable compatible almost complex structure $J_0$. Then there exists a compatible integrable almost complex structure $J$ on $(M,\omega)$ such that $\rho = \lambda \omega$.
\end{prop}
\begin{proof}
Fix the complex structure $J_0$ on $M$.
By Aubin's and Yau's results on K\"ahler-Einstien metrics of non-positive scalar curvature \cite{Aubin1978,Yau1978}, one can find a smooth path of cohomologous symplectic forms $\omega_t$ defined for $t \in [0,1]$ such that $\omega_0=\omega$, and $\omega_1$ satisfies $\rho_1 = \lambda \omega_1$.
By Moser's stability theorem, there exists a path of diffeomorphisms $\phi_t$ of $M$ such that $\phi_0=id$, and $\phi_t^*\omega_t = \omega$ for all $t \in [0,1]$. As a consequence, $J = d\phi_1^{-1} J_0 d\phi_1$ is an integrable complex structure as in the statement.
\end{proof}

After the celebrated work of Chen, Donaldson and Sun \cite{ChenDonaldsonSun2015} and Tian \cite{Tian2015}, we know that a Fano manifold admits a K\"ahler-Einstein metric if and only if it is K-polystable. 
Therefore, proposition above can be extended to the case $\lambda>0$ if in addition $J_0$ is assumed to be $K$-polystable. 

In view of this discussion, we are primarily interested in symplectic manifolds admitting no integrable compatible almost-complex structures at all, where the existence of special compatible almost complex structures is not known in general.

\section{Homogeneous symplectic manifolds}\label{sec::homsympman}

The general theory of homogeneous symplectic manifold has been
% initiated by Kostant \cite{Kostant1970} and Souriau \cite{Souriau1970} and 
developed by Chu \cite{Chu1974}. See also the work of Sternberg on this topic \cite{Sternberg1975}.
Here we consider a homogeneous symplectic manifold endowed with an additional structure, namely a homogeneous compatible almost complex structure. The aim of this section is showing how the Chern connection and several quantities related to it can be expressed in terms of algebraic structures on the relevant Lie algebra.

\subsection{General setup}\label{subsec::gensetup}

Let $(M,\omega)$ be a connected symplectic manifold endowed with a compatible almost complex structure $J$, and let $g$ be the associated Riemannian metric.
Assume that $\omega$ and $J$ are homogeneous, that is, there is a closed Lie group $G$ that acts transitively on $M$ by holomorphic and symplectic diffeomorphisms. This means that for all $\phi \in G$ it holds
\begin{equation}
\phi^*\omega = \omega, \qquad Jd\phi = d\phi J.
\end{equation}
From this it follows directly that any $\phi \in G$ is also an isometry of $g$.
Therefore $(M,g)$ turns out to be a $G$-homogeneous Riemannian manifold.

A good situation to have in mind is when $G$ is a closed subgroup of $\Diff(M)$ constituted by holomorphic and symplectic diffeomorphisms.
In this case, the action is effective.
This means that the identity is the only element of $G$ that acts as the identity on $M$.
Unfortunately, not all the applications we are interested in fall directly within this case.
Therefore, in order to match ease of exposition and flexibility for applications, we make the additional assumption that 
\begin{center}
the action of $G$ on $M$ is \emph{almost effective},
\end{center} meaning that there are no non-discrete normal subgroups of $G$ that act trivially on $M$. 
Under this assumption the isotropy subgroups need not to be compact, but we can always suppose they are.
To see this, fix once for all a point $x \in M$, and let $K \subset G$ be the isotropy subgroup at $x$, that is the set of all $\phi \in G$ such that $\phi(x)=x$.
Now consider the maximal normal subgroup $N$ of $G$ contained in $K$, and note that $N$ acts trivially on $M$. Thus $G'=G/N$ acts effectively on $M$ with isotropy at $x$ given by $K'=K/N$.
Therefore $M$ is diffeomorphic to both the coset spaces $G/K$ and $G'/K'$.
The advantage of introducing $K'$ is that it is compact whilst $K$ in general is not.
In fact, since $G'$ acts by isometries on $M$, the isotropy representation of $K'$ at $x$ is an isomorphism with a closed subgroup of the orthogonal group of $T_xM$.
Moreover, the kernel of the isotropy action of $K$ at $x$ is precisely $N$.
On the other hand, our assumption that $G$ acts almost effectively on $M$ guarantees that $N$ is discrete, whence the Lie algebras of $G$ and $K$ coincide with those of $G'$ and $K'$.
Since, as we will show, the almost-K\"ahler geometry of $(M,\omega,J)$ is completely described in terms of the Lie algebras of $G$ and $K$, and their adjoint actions, the upshot is that we can work with $G$ and $K$ as in case the action of $G$ were effective and $K$ were compact.

Let now $\mathfrak g$ be the Lie algebra of $G$. Given $X\in\mathfrak g$, consider the one-parameter group of diffeomorphisms of $M$ defined by
\begin{equation}
\phi_t(z) = e^{tX}z.
\end{equation}
We denote by $X$ the infinitesimal generator of $\phi_t$ as well.
Hence, our assumptions on $G$ imply that $X$ is a symplectic ($L_X\omega=0$), holomorphic ($L_XJ=0$) and Killing ($L_Xg=0$) vector field on $M$.
One can check that this identification of $\mathfrak g$ with a Lie subalgebra of vector fields on $M$ is an anti-homomorphism, that is the Lie bracket in $\mathfrak g$ and the Lie bracket of vector fields are related by 
\begin{equation}\label{eq::relationbetweenbrackets}
[X,Y]_{\mathfrak g} = -[X,Y].
\end{equation}

Now consider the Lie subalgebra $\mathfrak k$ of $K$ inside $\mathfrak g$.
Under the identification above, $\mathfrak k$ is constituted by vector fields $X$ which are symplectic and holomorphic and vanishes at $x$.
Consider the adjoint action of $K$ on $\mathfrak g$.
Since $K$ is compact one can find a $K$-invariant subspace $\mathfrak m \subset \mathfrak g$ such that 
\begin{equation}
\mathfrak g = \mathfrak k \oplus \mathfrak m.
\end{equation}
Fix once for all such an $\mathfrak m$. Evaluating at $x$ any vector field $X \in \mathfrak g$ gives an isomorphism between $\mathfrak m$ and $T_xM$ that intertwines the isotropy action of $K$ at $x$ with the adjoint action of $K$ on $\mathfrak m$.
Thanks to this ismomorphism one can endow $\mathfrak m$ with a symplectic form $\sigma$ and a complex structure $H$ by letting
\begin{equation}
\sigma(X,Y)=\omega_x(X,Y), \qquad HX=J_xX.
\end{equation}
Also a scalar product on $\mathfrak m$ is defined by letting $\langle X,Y \rangle = g_x(X,Y)$, and it satisfies $\langle X,Y \rangle=\sigma(X,HY)$ thanks to compatibility of $J$ with $\omega$.
Since $\omega$ and $J$ are $G$-invariant and $x$ is fixed by $K$, the symplectic form $\sigma$ and the complex structure $H$ are invariant with respect to the adjoint action of $K$ on $\mathfrak m$. As a consequence, the resulting scalar product on $\mathfrak m$ is $K$-invariant as well.
Moreover, the closedness of $\omega$ implies that
\begin{equation}\label{eq::closedsigma}
%\omega([X,Y],Z)+\omega(Y,[X,Z])
%-\omega([Y,X],Z)-\omega(X,[Y,Z])
%+\omega([Z,X],Y)+\omega(X,[Z,Y])
%-\omega([X,Y],Z)-\omega(Y,[X,Z])
%+\omega(X,[Y,Z])=0
%
%\omega([X,Y],Z)+\omega([Y,Z],X)+\omega([Z,X],Y)=0
\sigma([X,Y]_{\mathfrak m},Z)+\sigma([Y,Z]_{\mathfrak m},X)+\sigma([Z,X]_{\mathfrak m},Y)=0
\end{equation}
for all $X,Y,Z \in \mathfrak m$. Here $[X,Y]_{\mathfrak m}$ denotes the projection to $\mathfrak m$ of the Lie algebra bracket $[X,Y]_{\mathfrak g}$.
This follows by substituting the $G$-invariance condition
\begin{equation}
X\omega(Y,Z)=\omega([X,Y],Z)+\omega(Y,[X,Z]).
\end{equation}
into the closedness condition
\begin{multline}
X\omega(Y,Z)-Y\omega(X,Z)+Z\omega(X,Y) \\
-\omega([X,Y],Z)-\omega(Y,[X,Z])
+\omega(X,[Y,Z])=0,
\end{multline}
and evaluating at the point $x$.

\subsection{Chern-Ricci form}

The $G$-invariance of the Riemannian metric $g$ implies the $G$-invariance of the Riemann curvature.
In the same way, the $G$-invariance of $\omega$ and $J$ imply the $G$-invariance of the Nijenhuis tensor $N$ and the Chern curvature $R$ of $J$.
As a consequence both of them can be expressed in terms of the symplectic form $\sigma$, the complex structure $H$ on $\mathfrak m$, and the Lie algebra structure of $\mathfrak g$.
Here and in the next sub-section we will show how $N$, the Chern-Ricci form $\rho$, and the Hermitian scalar curvature are related to $\sigma$ and $H$.

We start with the following lemma, which is well known \cite[Poposition 7.28]{Besse1987}, and we include here a proof for convenience of the reader.
\begin{lem}\label{lem:LeviCivitasymm}
Let $X,Y,Z$ be symplectic holomorphic vector fields in $\mathfrak m$. Then at the point $x$ the Levi-Civita connection $D$ of $g$ satisfies
\begin{equation}\label{eq::homLeviCivita}
(D_XY)_x = -\frac{1}{2}[X,Y]_{\mathfrak m} + U(X,Y),
\end{equation}
where $U$ is the symmetric and bilinear binary operation on $\mathfrak m$ defined by
\begin{equation}
\langle U(X,Y),Z\rangle = \langle[Z,X]_{\mathfrak m},Y\rangle + \langle X,[Z,Y]_{\mathfrak m}\rangle.
\end{equation}
\end{lem}
\begin{proof}
Since $X,Y,Z$ are symplectic and holomorphic, they are Killing as well.
Therefore
\begin{equation}
g([X,Z],X) = g(D_XZ,X) - g(D_ZX,X) = g(D_XX,Z),
\end{equation}
whence
\begin{equation}
g(D_XY,Z)+g(D_YX,Z) = g([X,Z],Y) + g([Y,Z],X),
\end{equation}
and
\begin{equation}
2g(D_XY,Z) = g([X,Y],Z) + g([X,Z],Y) + g([Y,Z],X).
\end{equation}
At $x$ one the has
\begin{equation}
\langle (D_XY)_x,Z \rangle = -\frac{1}{2} \langle [X,Y]_{\mathfrak m},Z \rangle -\frac{1}{2} \langle [X,Z]_{\mathfrak m},Y \rangle -\frac{1}{2} \langle[Y,Z]_{\mathfrak m},X\rangle,
\end{equation}
whence the thesis follows.
\end{proof}
Thanks to lemma above, we can express the covariant derivative of $J$ with respect to the Levi-Civita connection $D$ in terms of $H$ and the bilinear form $U$.
\begin{cor}\label{cor::DJsymm}
Let $X,Y$ be symplectic holomorphic vector fields in $\mathfrak m$. Then at the point $x$ one has
\begin{equation}
(D_XJ)_xY = \frac{1}{2}[X,HY]_{\mathfrak m} - \frac{1}{2}H[X,Y]_{\mathfrak m} + U(X,HY) - HU(X,Y).
\end{equation}
\end{cor}
\begin{proof}
For any vector field $Z \in \mathfrak m$ one has
\begin{equation}
g\left((D_XJ)Y,Z\right) 
= g\left(D_X(JY),Z\right) - g\left(JD_XY,Z\right).
\end{equation}
Since $X$ is holomorphic, and $D$ is torsion free, the above equation reduces to
\begin{equation}
g\left((D_XJ)Y,Z\right) 
%= g\left(D_{JY}X,Z\right) + g\left([X,JY],Z\right) - g\left(JD_YX,Z\right) - g\left(J[X,Y],Z\right)
= g\left(D_{JY}X,Z\right)-g\left(JD_YX,Z\right),
\end{equation}
whence the thesis follows readily by lemma \ref{lem:LeviCivitasymm}.
\end{proof}
As a consequence of corollary above, one gets an expression of the Nijenhuis torsion $N$ at $x$ in terms of $H$.
\begin{lem}\label{lem::Nijenhuishomsymm}
Let $X,Y$ be symplectic holomorphic vector fields in $\mathfrak m$. Then at the point $x$ one has
\begin{equation}
N(X,Y)_x = N_H(X,Y),
\end{equation}
where $4N_H(X,Y) = [HX,HY]_{\mathfrak m}-H[HX,Y]_{\mathfrak m}-H[X,HY]_{\mathfrak m}-[X,Y]_{\mathfrak m}$ defines the Nijenhuis tensor of $H$.
\end{lem}
\begin{proof}
Since the Levi-Civita connection $D$ is torsion free, one has the identity
\begin{equation}
2N(X,Y) = J(D_YJ)X - J(D_XJ)Y.
\end{equation}
Therefore, thanks to corollary \ref{cor::DJsymm}, evaluating $2g\left(N(X,Y),JZ\right)$ at $x$ yields
\begin{multline}
2\langle N(X,Y)_x,HZ\rangle = 
2\langle N_H(X,Y),HZ\rangle \\
+ \frac{1}{2} \sigma\left([X,Y]_{\mathfrak m},Z\right) 
+ \frac{1}{2} \sigma\left( [Y,Z]_{\mathfrak m},X\right)
+ \frac{1}{2} \sigma\left( [Z,X]_{\mathfrak m},Y\right) \\
+ \frac{1}{2} \sigma\left([HX,HY]_{\mathfrak m},Z\right)
+\frac{1}{2} \sigma\left( [HY,Z]_{\mathfrak m},HX\right)
+ \frac{1}{2} \sigma\left( [Z,HX]_{\mathfrak m},HY\right).
\end{multline}
By identity \eqref{eq::closedsigma} the second and third line of equation above vanish.
Hence the thesis follows by arbitrariness of $Z$.
\end{proof}
As we did for the Levi-Civita connection above, we can express the covariant derivative at $x$ of vector fields in $\mathfrak m$ with respect to the Chern connection in term of the Lie algebra structure of $\mathfrak g$.
\begin{lem}\label{lem::Chernconnectionsymm}
Let $X,Y,Z$ be symplectic holomorphic vector fields in $\mathfrak m$. Then at the point $x$ the Chern connection $\nabla$ of $J$ satisfies
%\begin{equation}
%(\nabla_XY)_x = -\frac{3}{4}[X,Y]_{\mathfrak m} 
%-\frac{1}{4}H[X,HY]_{\mathfrak m} + \frac{1}{2} U(X,Y) -\frac{1}{2} HU(X,HY)
%\end{equation} 
\begin{equation}
(\nabla_XY)_x = -\frac{1}{2}[X,Y]_{\mathfrak m} + V(X,Y),
\end{equation}
where $V:\mathfrak m \times \mathfrak m \to \mathfrak m$ is the bilinear binary operation on $\mathfrak m$ defined by
\begin{equation}
\langle V(X,Y),Z \rangle = \langle U(X,Y),Z \rangle - \langle X,N_H(Y,Z) \rangle.
\end{equation}
\end{lem}
\begin{proof}
By lemma \ref{lem:LeviCivitasymm}, evaluating at $x$ the identity $\nabla_XY = D_XY - \frac{1}{2}J(D_XJ)Y$ yields
\begin{equation}
(\nabla_XY)_x = -\frac{1}{2}[X,Y]_{\mathfrak m} + U(X,Y) - \frac{1}{2} H(D_XJ)_xY.
\end{equation}
Therefore we are reduced to prove the identity
\begin{equation}\label{eq::crucialidentity}
\sigma((D_XJ)_xY, Z ) = \langle X,2N_H(Y,Z) \rangle.
\end{equation}
In order to do this, by corollary \eqref{cor::DJsymm} and the definition of $U$ we compute
%\begin{eqnarray*}
%\sigma((D_XJ)_xY, Z ) 
%&=&
%\frac{1}{2} \sigma([X,HY]_{\mathfrak m},Z) - \frac{1}{2} \sigma(H[X,Y]_{\mathfrak m},Z) \\
%&& + \sigma(U(X,HY),Z) - \sigma(HU(X,Y),Z) \\
%&=&
%\frac{1}{2} \sigma([X,HY]_{\mathfrak m},Z) - \frac{1}{2} \sigma(H[X,Y]_{\mathfrak m},Z) \\
%&& - \langle U(X,HY),HZ\rangle + \langle U(X,Y),Z \rangle \\
%&=&
%\frac{1}{2}\sigma([X,HY]_{\mathfrak m},Z) - \frac{1}{2}\sigma(H[X,Y]_{\mathfrak m},Z) \\
%&& - \frac{1}{2} \langle[HZ,X]_{\mathfrak m},HY\rangle - \frac{1}{2} \langle X,[HZ,HY]_{\mathfrak m}\rangle \\
%&& + \frac{1}{2} \langle[Z,X]_{\mathfrak m},Y\rangle + \frac{1}{2} \langle X,[Z,Y]_{\mathfrak m}\rangle \\
%&=&
%\frac{1}{2}\sigma([X,HY]_{\mathfrak m},Z) + \frac{1}{2}\sigma([X,Y]_{\mathfrak m},HZ) \\
%&& + \frac{1}{2} \sigma([HZ,X]_{\mathfrak m},Y) - \frac{1}{2} \sigma(X,H[HZ,HY]_{\mathfrak m}) \\
%&& + \frac{1}{2} \sigma([Z,X]_{\mathfrak m},HY) + \frac{1}{2} \sigma(X,H[Z,Y]_{\mathfrak m}) \\
%&=&
%- \frac{1}{2}\sigma([HY,Z]_{\mathfrak m},X) - \frac{1}{2} \sigma([Y,HZ]_{\mathfrak m},X) \\
%&& - \frac{1}{2} \sigma(X,H[HZ,HY]_{\mathfrak m}) + \frac{1}{2} \sigma(X,H[Z,Y]_{\mathfrak m}) \\
%&=&
%- \frac{1}{2}\sigma(X,[Z,HY]_{\mathfrak m}) - \frac{1}{2} \sigma(X,[HZ,Y]_{\mathfrak m}) \\
%&& - \frac{1}{2} \sigma(X,H[HZ,HY]_{\mathfrak m}) + \frac{1}{2} \sigma(X,H[Z,Y]_{\mathfrak m}) \\
%&=& - \sigma(X,2HN_H(Z,Y))
%\end{eqnarray*}
\begin{multline}
\sigma((D_XJ)_xY, Z ) = 
\frac{1}{2}\sigma([X,HY]_{\mathfrak m},Z) 
- \frac{1}{2}\sigma(H[X,Y]_{\mathfrak m},Z) 
+ \frac{1}{2} \sigma([HZ,X]_{\mathfrak m},Y) \\
- \frac{1}{2} \sigma(X,H[HZ,HY]_{\mathfrak m}) 
+ \frac{1}{2} \sigma([Z,X]_{\mathfrak m},HY) 
+ \frac{1}{2} \sigma(X,H[Z,Y]_{\mathfrak m}),
\end{multline}
whence \eqref{eq::crucialidentity} follows by identity \eqref{eq::closedsigma} and definition of $N_H$.
\end{proof}
Now we are in position to prove the main result of this section.
\begin{thm}\label{thm::chernriccionm}
Let $X,Y$ be symplectic and holomorphic vector fields in $\mathfrak m$.
At the point $x$ the Chern-Ricci form of $J$ is given by
\begin{equation}
\rho(X,Y)_x = \sum_{i=1}^{2n} \langle [H[X,Y]_{\mathfrak m},e_i]_{\mathfrak m},e_i \rangle
- \langle H[[X,Y]_{\mathfrak g},e_i]_{\mathfrak m},e_i\rangle,
\end{equation}
where $e_1,\dots,e_{2n}$ is any orthonormal basis of $\mathfrak m$. 
\end{thm}
\begin{proof}
Denoted as above by $R$ the curvature of the Chern connection $\nabla$, the Chern-Ricci form of $J$ is given by
\begin{equation}
\rho(X,Y) = \tr\left(JR(X,Y)\right).
\end{equation}
Since the formula in the statement does not depend on the particular orthonormal basis $e_1,\dots,e_{2n} \in \mathfrak m$ chosen, we can furthermore assume with no loss that $e_{i+n}=He_i$ for all $i \in \{1,\dots,n\}$.
Therefore the smooth function
%\begin{eqnarray*}
%\rho(X,Y)
%&=& \sum_{i=1}^{2n} \langle JR(X,Y)e_i,e_i \rangle \\
%&=& 2 \sum_{i=1}^n \langle JR(X,Y)e_i,e_i \rangle \\
%&=& 2 \sum_{i=1}^n \omega( R(X,Y)e_i,e_i )
%\end{eqnarray*}
\begin{equation}
f = \sum_{i=1}^n 2\omega(R(X,Y)e_i,e_i).
\end{equation}
satisfies $\rho(X,Y)_x = f(x)$.
Since $X,Y$ are symplectic and holomorphic, one has
\begin{eqnarray*}
f
%&=& 2 \sum_{i=1}^n \omega(R(X,Y)e_i,e_i) \\
%&=& 2 \sum_{i=1}^n \omega(\nabla_X\nabla_Ye_i,e_i)
%- \omega(\nabla_Y\nabla_Xe_i,e_i)
%- \omega(\nabla_{[X,Y]}e_i,e_i) \\
%&=& 2 \sum_{i=1}^n 
%X\left(\omega(\nabla_Ye_i,e_i)\right)
%- Y\left(\omega(\nabla_Xe_i,e_i)\right) \\
%&& + 2\omega(\nabla_Xe_i,\nabla_Ye_i)
%- \omega(\nabla_{[X,Y]}e_i,e_i) \\
&=& 2 \sum_{i=1}^n 
\omega([X,\nabla_Ye_i],e_i)
+ \omega(\nabla_Ye_i,[X,e_i])
- \omega([Y,\nabla_Xe_i],e_i) \\
&& - \omega(\nabla_Xe_i,[Y,e_i])
+ 2\omega(\nabla_Xe_i,\nabla_Ye_i)
- \omega(\nabla_{[X,Y]}e_i,e_i).
\end{eqnarray*}
By lemma \ref{lem::LiederCherncon} then the equation above reduces to
%\begin{eqnarray*}
%\rho(X,Y)
%&=& 2 \sum_{i=1}^n 
%\omega(\nabla_{[X,Y]}e_i,e_i)
%+ \omega(\nabla_Y[X,e_i],e_i)
%+ \omega(\nabla_Ye_i,[X,e_i]) \\
%&& - \omega(\nabla_{[Y,X]}e_i,e_i)
%- \omega(\nabla_X[Y,e_i],e_i)
%- \omega(\nabla_Xe_i,[Y,e_i]) \\
%&& + 2\omega(\nabla_Xe_i,\nabla_Ye_i)
%- \omega(\nabla_{[X,Y]}e_i,e_i) \\
%&=& 2 \sum_{i=1}^n 
%\omega(\nabla_{[X,Y]}e_i,e_i)
%+ Y \left(\omega([X,e_i],e_i)\right)
%+ 2 \omega(\nabla_Ye_i,[X,e_i]) \\
%&& - X \left(\omega([Y,e_i],e_i)\right)
%- 2 \omega(\nabla_Xe_i,[Y,e_i])
%+ 2\omega(\nabla_Xe_i,\nabla_Ye_i) \\
%&=& 2 \sum_{i=1}^n 
%2\omega(\nabla_Xe_i,\nabla_Ye_i)
%+ \omega(\nabla_{[X,Y]}e_i,e_i) \\
%&& + \omega([Y,[X,e_i]],e_i)
%+ \omega([X,e_i],[Y,e_i])
%+ 2 \omega(\nabla_Ye_i,[X,e_i]) \\
%&& - \omega([X,[Y,e_i]],e_i)
%- \omega([Y,e_i],[X,e_i])
%- 2 \omega(\nabla_Xe_i,[Y,e_i]) \\
%&=& 2 \sum_{i=1}^n 
%2\omega(\nabla_Xe_i-[X,e_i],\nabla_Ye_i-[Y,e_i]) \\
%&& + \omega(\nabla_{[X,Y]}e_i-[[X,Y],e_i],e_i)
%\end{eqnarray*}
\begin{multline}
f = \sum_{i=1}^n 
4\omega(\nabla_Xe_i-[X,e_i],\nabla_Ye_i-[Y,e_i]) \\
+ 2\omega(\nabla_{[X,Y]}e_i-[[X,Y],e_i],e_i).
\end{multline}
Therefore by lemma \ref{lem::Chernconnectionsymm}, evaluating at $x$ yields
%\begin{eqnarray*}
%\rho(X,Y)_x 
%&=& \sum_{i=1}^n 
%4\sigma\left(-\frac{1}{2}[X,e_i]_{\mathfrak m}+V(X,e_i)+[X,e_i]_{\mathfrak m},-\frac{1}{2}[Y,e_i]_{\mathfrak m}+V(Y,e_i)+[Y,e_i]_{\mathfrak m}\right) \\
%&& + 2\sigma\left(\frac{1}{2}[[X,Y]_{\mathfrak m},e_i]_{\mathfrak m}-V([X,Y]_{\mathfrak m},e_i)-[[X,Y]_{\mathfrak g},e_i]_{\mathfrak m},e_i\right) \\
%&=& \sum_{i=1}^n 
%\sigma([X,e_i]_{\mathfrak m} + 2V(X,e_i),[Y,e_i]_{\mathfrak m} + 2V(Y,e_i)) \\
%&& + \sigma([[X,Y]_{\mathfrak m},e_i]_{\mathfrak m} - 2V([X,Y]_{\mathfrak m},e_i),e_i)
%- 2\sigma\left([[X,Y]_{\mathfrak g},e_i]_{\mathfrak m},e_i\right) \\
%&=& \sum_{i=1}^n 
%\sigma([X,e_i]_{\mathfrak m} + 2V(X,e_i),[Y,e_i]_{\mathfrak m} + 2V(Y,e_i)) \\
%&& - \sigma([[X,Y]_{\mathfrak m},e_i]_{\mathfrak m} + 2V([X,Y]_{\mathfrak m},e_i),e_i)
%- 2\sigma\left([[X,Y]_{\mathfrak k},e_i]_{\mathfrak m},e_i\right) \\
%\end{eqnarray*}
\begin{multline}
\rho(X,Y)_x = \sum_{i=1}^n 
\sigma([X,e_i]_{\mathfrak m} + 2V(X,e_i),[Y,e_i]_{\mathfrak m} + 2V(Y,e_i)) \\
- \sigma([[X,Y]_{\mathfrak m},e_i]_{\mathfrak m} + 2V([X,Y]_{\mathfrak m},e_i),e_i)
- 2\sigma\left([[X,Y]_{\mathfrak k},e_i]_{\mathfrak m},e_i\right).
\end{multline}
In order to simplify further the identity above, note that 
\begin{multline}\label{eq::rhoatxnonyetsimplified}
\rho(X,Y)_x = \sum_{i=1}^n 
\sigma(W(X,e_i),W(Y,e_i))
- \sigma(W([X,Y]_{\mathfrak m},e_i),e_i) \\
- 2\sigma\left([[X,Y]_{\mathfrak k},e_i]_{\mathfrak m},e_i\right).
\end{multline}
where $W: \mathfrak m \times \mathfrak m \to \mathfrak m$ is a bilinear map defined by
\begin{equation}
W(X,Y) = [X,Y]_{\mathfrak m} + 2V(X,Y).
\end{equation}
By definition of $V$ it follows that
\begin{equation}
W(X,Y) = [X,Y]_{\mathfrak m} + 2 U(X,Y) - \sum_{j=1}^{2n} 2\langle X,N_H(Y,e_j) \rangle e_j.
\end{equation}
On the other hand, substituting the identity \eqref{eq::closedsigma} into the definition of $U$ yields
%\begin{eqnarray*}
%2U(X,Y)
%&=& \sum_{i=1}^{2n}
%\langle[e_i,X]_{\mathfrak m},Y\rangle e_i + \langle X,[e_i,Y]_{\mathfrak m}\rangle e_i \\
%&=& \sum_{i=1}^{2n}
%\langle[e_i,X]_{\mathfrak m},Y\rangle e_i - \sigma( HX,[e_i,Y]_{\mathfrak m}) e_i \\
%&=& \sum_{i=1}^{2n}
%\langle[e_i,X]_{\mathfrak m},Y\rangle e_i 
%+ \sigma( e_i,[Y,HX]_{\mathfrak m}) e_i
%+ \sigma( Y,[HX,e_i]_{\mathfrak m}) e_i \\
%&=& \sum_{i=1}^{2n}
%\langle[e_i,X]_{\mathfrak m},Y\rangle e_i 
%+ \langle He_i,[Y,HX]_{\mathfrak m}\rangle e_i
%+ \langle HY,[HX,e_i]_{\mathfrak m}\rangle e_i \\
%&=& \sum_{i=1}^{2n}
%\langle[e_i,X]_{\mathfrak m},Y\rangle e_i 
%- \langle e_i,H[Y,HX]_{\mathfrak m}\rangle e_i
%- \langle Y,H[HX,e_i]_{\mathfrak m}\rangle e_i
%\end{eqnarray*}
\begin{equation}
2U(X,Y)
= H[HX,Y]_{\mathfrak m}
- \sum_{j=1}^{2n}
\langle[X,e_j]_{\mathfrak m}+H[HX,e_j]_{\mathfrak m},Y\rangle e_j,
\end{equation}
whence
\begin{multline}
W(X,Y) = [X,Y]_{\mathfrak m} + H[HX,Y]_{\mathfrak m}
- \sum_{j=1}^{2n}
2\langle X,N_H(Y,e_j) \rangle e_j \\
+ \langle[X,e_j]_{\mathfrak m}+H[HX,e_j]_{\mathfrak m},Y\rangle e_j.
\end{multline}
Thanks to the following identity, which is a direct consequence of closedness of $\sigma$ and definition of $N_H$,
\begin{equation}\label{eq::cyclicNijenhuis}
\langle X, N_H(Y,e_j) \rangle + \langle Y, N_H(e_j,X) \rangle + \langle e_j, N_H(X,Y) \rangle = 0
\end{equation}
one finally has
\begin{multline}
W(X,Y) = 2N_H(X,Y) + [X,Y]_{\mathfrak m} + H[HX,Y]_{\mathfrak m} \\
- \sum_{j=1}^{2n} \langle 2N_H(X,e_j) + [X,e_j]_{\mathfrak m}+H[HX,e_j]_{\mathfrak m},Y\rangle e_j.
\end{multline}
Therefore, letting $f_X(Y) = 2N_H(X,Y) + [X,Y]_{\mathfrak m} + H[HX,Y]_{\mathfrak m}$ defines an endomorphism $f_X$ of $\mathfrak m$ such that
\begin{equation}
W(X,Y) = f_X(Y) - f_X^*(Y), 
\end{equation}
where $f_X^*$ denotes the adjoint endomorphism.
By definition of $N_H$ it is easy to check that $f_X$ and $f_X^*$ commute with $H$, whence it follows the identity
\begin{equation}
\sigma(f_X(Y),Z) = \sigma(Y,f_X^*(Z)).
\end{equation}
As a consequence, \eqref{eq::rhoatxnonyetsimplified} reduces to
%\begin{eqnarray*}
%\sum_{i=1}^n \sigma \left( W(X,e_i),W(Y,e_i) \right)
%&=& \sum_{i=1}^n \sigma \left( f_X(e_i)-f_X^*(e_i),f_Y(e_i)-f_Y^*(e_i) \right) \\
%&=& \sum_{i=1}^n \sigma \left( [f_X,f_Y](e_i),e_i \right) + \sigma \left( [f_Y^*,f_X](e_i),e_i \right) \\
%&=& \sum_{i=1}^n \sigma \left( [f_X,f_Y-f_Y^*](e_i),e_i \right)
%\end{eqnarray*}
\begin{multline}
\rho(X,Y)_x = \sum_{i=1}^n \sigma \left( [f_X,f_Y-f_Y^*](e_i),e_i \right)
- 2 \sigma(f_{[X,Y]_{\mathfrak m}}(e_i),e_i) \\
- 2\sigma\left([[X,Y]_{\mathfrak k},e_i]_{\mathfrak m},e_i\right).
\end{multline}
By commuting relations noticed above, one has
\begin{equation}
\sum_{i=1}^n \sigma \left( [f_X,f_Y-f_Y^*](e_i),e_i \right)
= \frac{1}{2} \tr\left( [Hf_X,f_Y-f_Y^*] \right) = 0,
\end{equation}
whence
\begin{equation}
\rho(X,Y)_x = - \sum_{i=1}^{2n}
\sigma(f_{[X,Y]_{\mathfrak m}}(e_i),e_i)
+ \sigma\left([[X,Y]_{\mathfrak k},e_i]_{\mathfrak m},e_i\right).
\end{equation}
The thesis then follows by definition of $f_{[X,Y]_{\mathfrak m}}$ and by identity \eqref{eq::cyclicNijenhuis}.
%\begin{eqnarray*}
%\rho(X,Y)_x 
%&=& - \sum_{i=1}^{2n}
%\sigma(f_{[X,Y]_{\mathfrak m}}(e_i),e_i)
%+ \sigma\left([[X,Y]_{\mathfrak k},e_i]_{\mathfrak m},e_i\right) \\
%&=& - \sum_{i=1}^{2n}
%2\sigma(N_H([X,Y]_{\mathfrak m},e_i),e_i)
%+ \sigma([[X,Y]_{\mathfrak m},e_i]_{\mathfrak m},e_i) \\
%&& + \sigma(H[H[X,Y]_{\mathfrak m},e_i]_{\mathfrak m},e_i)
%+ \sigma\left([[X,Y]_{\mathfrak k},e_i]_{\mathfrak m},e_i\right) \\
%&=& \sum_{i=1}^{2n} \sigma([H[X,Y]_{\mathfrak m},e_i]_{\mathfrak m},He_i)
%- \sigma\left([[X,Y]_{\mathfrak g},e_i]_{\mathfrak m},e_i\right) \\
%&=& \sum_{i=1}^{2n} \langle [H[X,Y]_{\mathfrak m},e_i]_{\mathfrak m},e_i \rangle
%- \langle H[[X,Y]_{\mathfrak g},e_i]_{\mathfrak m},e_i\rangle
%\end{eqnarray*}
\end{proof}
The above formula for the Chern-Ricci form simplifies notationally a bit after extending $H$ to the whole $\mathfrak g$ by defining $H\mathfrak k = \{0\}$, and considering the adjoint action of $\mathfrak g$ on itself denoted by $\ad_X(Y) = [X,Y]_{\mathfrak g}$. As a consequence one readily has the following
\begin{cor}\label{cor::rhohomwithW}
Let $X,Y$ be symplectic and holomorphic vector fields in $\mathfrak m$.
The Chern-Ricci form of $J$ at $x$ satisfies
\begin{equation}
\rho(X,Y)_x = \tr \left( \ad_{H[X,Y]_{\mathfrak g}} - H \ad_{[X,Y]_{\mathfrak g}} \right).
\end{equation}
\end{cor}
This result suggests to define a linear one-form $\zeta \in \mathfrak g^*$ by letting
\begin{equation}\label{eq::defzeta}
\zeta(X) = \tr \left( \ad_{HX} - H \ad_{X} \right).
\end{equation}
By straightforward calculation involving $K$-invariance of $H$ and basic facts about adjoint representations, one can check that $\zeta$ is $K$-invariant.
%\begin{eqnarray*}
%\zeta(\Ad_k(X))
%&=& \tr \left( \ad_{H\Ad_k(X)} - H \ad_{\Ad_k(X)} \right) \\
%&=& \tr \left( \ad_{\Ad_k(HX)} - H \ad_{\Ad_k(X)} \right) \\
%&=& \tr \left( \Ad_k\ad_{HX}\Ad_k^{-1} - H \Ad_k\ad_X\Ad_k^{-1} \right) \\
%&=& \tr \left( \Ad_k\ad_{HX}\Ad_k^{-1} - \Ad_k H \ad_X\Ad_k^{-1} \right) \\
%&=& \tr \left( \ad_{HX} - H \ad_X \right) \\
%&=& \zeta(X)
%\end{eqnarray*}
As a consequence, the restriction of $\zeta$ to $\mathfrak k$ is completely determined by the restriction of $\zeta$ to the center of $\mathfrak k$.
This follows from the fact that, by compactness, $\mathfrak k$ decomposes as the sum of its center $\mathfrak z$ and a semi-simple subalgebra $\mathfrak s$. 
Since any element of $\mathfrak s$ can be written in the form $[X,Y]$ with $X,Y \in \mathfrak s$, $K$-invariance of $\zeta$ implies that $\mathfrak s$ is contained in the kernel of $\zeta$.
Having introduced $\zeta$, an easy but important consequence of corollary above is that the Chern-Ricci form of $J$ is determined by $\zeta$ via the identity
\begin{equation}\label{eq::rhoatx}
\rho(X,Y)_x = \zeta([X,Y]_{\mathfrak m}).
\end{equation}
This formula suggests that a symplectic manifold admitting a homogeneous special compatible almost complex structure with non-zero Hermitian scalar curvature is, up to coverings, a co-adjoint orbit equipped with the Kirillov-Kostant-Souriau symplectic form.
We will enter into details of this in sub-section \ref{subsect::coadjointorbits}.

\subsection{Hermitian scalar curvature}\label{subsec::hermscalcurv}

Now we pass to consider the Hermitian scalar curvature $s$ of $J$.
Since $\omega$ and $J$ are homogeneous, then $s$ turns out to be a $G$-invariant function on $M$, hence constant. Moreover one has the following
\begin{cor}\label{cor::Hermitianscalarcurvhom}
The Hermitian scalar curvature of $J$ is given by
\begin{equation}\label{eq::hermscalcurvhomo}
s = \tr \left( \ad_{H\xi} - H \ad_\xi \right)
\end{equation}
where $\xi = \sum_{i=1}^n [e_i,e_{i+n}]_{\mathfrak g}$, and $\{e_i\}$ is any symplectic basis of $\mathfrak m$.
\end{cor}
\begin{proof}
As we already noticed above, $s$ is constant.
Therefore it is enough to evaluate it at $x$.
Given a symplectic basis $e_1,\dots,e_{2n}$ of $\mathfrak m$, by definition of the Hermitian scalar curvature one then has $s = \sum_{i=1}^n \rho(e_i,e_{i+n})_x$.
Hence formula \eqref{eq::hermscalcurvhomo} readily follows by corollary \ref{cor::rhohomwithW}.

Now it remains to show that $\xi$ does not depend on the chosen symplectic basis.
To this end let $\{f_i\}$ be another symplectic basis of $\mathfrak m$, and let $A$ be the symplectic endomorphism of $\mathfrak m$ such that $f_i=Ae_i$, for all $i\in\{1,\dots,2n\}$.
In the basis $\{e_i\}$ then $A$ is represented by a $2n \times 2n$ real matrix
\begin{equation}
\left(
\begin{array}{cc}
a & b \\
c & d
\end{array}
\right)
\end{equation}
where $a,b,c,d$ are $n \times n$ matrices satisfying $ab^t=ba^t$, $cd^t=dc^t$, $ad^t-bc^t=\id$.
Therefore $f_i = \sum_{j=1}^n a^j_i e_j + c^j_i e_{j+n}$, and $f_{i+n} = \sum_{h=1}^n b^h_i e_h + d^h_i e_{h+n}$,
whence
%\begin{eqnarray*}
%\sum_{i=1}^n [f_i,f_{i+n}]
%&=& \sum_{i,j,h=1}^n [a^j_i e_j + c^j_i e_{j+n},b^h_i e_h + d^h_i e_{h+n}] \\
%&=& \sum_{i,j,h=1}^n 
%a^j_i b^h_i  [e_j,e_h]
%+ a^j_i d^h_i [e_j,e_{h+n}]
%+ c^j_i b^h_i [e_{j+n},e_h]
%+ c^j_i d^h_i [e_{j+n},e_{h+n}] \\
%&=& \sum_{i,j,h=1}^n 
%a^j_i b^h_i  [e_j,e_h]
%+ (a^j_i d^h_i - c^h_i b^j_i) [e_j,e_{h+n}]
%+ c^j_i d^h_i [e_{j+n},e_{h+n}] \\
%&=& \sum_{i=1}^n [e_i,e_{i+n}]
%\end{eqnarray*}
\begin{multline}
\sum_{i=1}^n [f_i,f_{i+n}]
= \sum_{i,j,h=1}^n 
a^j_i b^h_i  [e_j,e_h]
+ (a^j_i d^h_i - c^h_i b^j_i) [e_j,e_{h+n}] \\
+ c^j_i d^h_i [e_{j+n},e_{h+n}]
= \sum_{i=1}^n [e_i,e_{i+n}].
\end{multline}
\end{proof}

The element $\xi \in \mathfrak g$ defined in the statement above depends just on the symplectic structure of $\mathfrak m$.
Moreover, $\xi$ is $K$-invariant. 
This follows from the fact that given $k \in K$ and a symplectic basis $\{e_i\}$ of $\mathfrak m$, $\{\Ad_k(e_i)\}$ is also a symplectic basis of $\mathfrak m$.
Now consider
\begin{equation}\label{eq::xidecomposition}
\xi = \xi_{\mathfrak k} + \xi_{\mathfrak m}
\end{equation}
according to decomposition $\mathfrak g = \mathfrak k \oplus \mathfrak m$.
Both $\xi_{\mathfrak k}$ and $\xi_{\mathfrak m}$ are $K$-invariant.
In particular, $\xi_{\mathfrak k}$ belongs to the center of $\mathfrak k$. 
On the other hand, note that for all $X \in \mathfrak m$ one has $\ad_X(\mathfrak k) \subset \mathfrak m$, whence
\begin{equation}
\tr(\ad_X) = \sum_{i=1}^n \sigma([X,e_i],He_i) - \sigma([X,He_i],e_i),
\end{equation}
being $e_1,\dots,e_n$ a unitary basis of $\mathfrak m$, that is one for which $\{e_i,He_i\}$ is a symplectic basis.
By closedness of $\sigma$ equation above simplifies to
%\begin{eqnarray*}
%\tr(\ad_X) 
%&=& - \sum_{i=1}^n \sigma([e_i,X],He_i) + \sigma([X,He_i],e_i) \\
%&=& \sum_{i=1}^n \sigma([He_i,e_i],X) \\
%&=& - \sigma(\xi_{\mathfrak m},X) \\
%&=& \sigma(X,\xi_{\mathfrak m})
%\end{eqnarray*}
\begin{equation}\label{eq::tradXviaxim}
\tr(\ad_X) = \sigma(X,\xi_{\mathfrak m}).
\end{equation}
On the other hand, if $X \in \mathfrak k$, then $\ad_X$ is skew-invariant with respect to any $K$-invariant scalar product on $\mathfrak g$.
Such a scalar product always exists by compactness of $K$, hence $\tr(\ad_X)=0$.
At this point recall that $G$, which we assumed to be connected, is said to be unimodular if $\ad_X$ is trace-free for all $X \in \mathfrak g$.
Therefore, we can conclude that $\xi_{\mathfrak m}$ constitutes an obstruction to unimodularity of $G$.
We summarize the information on $\xi$ in the following 
\begin{prop}\label{prop::propertiesxi}
Let $\{e_i\}$ be a symplectic basis of $\mathfrak m$.
The element 
\begin{equation}
\xi = \sum_{i=1}^n [e_i,e_{i+n}] \in \mathfrak g
\end{equation}
does not depend on the chosen basis of $\mathfrak m$ and it is $K$-invariant.
The component of $\xi$ along $\mathfrak k$ belongs to the center of $\mathfrak k$.
The component of $\xi$ along $\mathfrak m$ is zero if and only if $G$ is unimodular.
\end{prop}

By homogeneity, the squared norm of the Nijenhuis tensor is constant.
More precisely we have the following
\begin{prop}\label{prop::formulaNijenhuis}
The squared norm of the Nijenhuis tensor of $J$ is given by
\begin{equation}
|N|^2 =\frac{1}{8} \sum_{i,j=1}^{2n} |[e_i,e_j]_{\mathfrak m}|^2 + \frac{1}{4} \sum_{i=1}^{2n} \tr(\ad_{e_i}^2) - \frac{1}{2} \tr(H\ad_\xi)
\end{equation}
where $\{e_i\}$ is any orthonormal basis of $\mathfrak m$. 
\end{prop}
\begin{proof}
Since $|N|^2$ is constant, it is enough to compute it at the point $x$, where we know that it is equal to the squared norm of $N_H$.
Therefore formula in the statement can be deduced by direct computation.
In doing this it could be useful assuming that $\{e_i\}$ is also a symplectic basis and than observing that the final formula actually does not requires this additional assumption.

On the other hand, the statement can be deduced by a combination of \eqref{eq::s=scal+2snormN} together with a well known formula for scalar curvature on homogeneous Riemannian manifold which reads \cite[Corollary 7.39]{Besse1987}:
\begin{equation}
\scal = - \frac{1}{4} \sum_{i,j=1}^{2n} |[e_i,e_j]_{\mathfrak m}|^2 - \frac{1}{2} \sum_{i=1}^{2n} \tr(\ad_{e_i}^2) - |Z|^2.
\end{equation}
Here $Z$ is the unique element of $\mathfrak m$ such that $\tr(\ad_X)= \langle Z,X \rangle$ for all $X \in \mathfrak m$.
After noting that in our situation equation \eqref{eq::tradXviaxim} yields $Z=-H\xi \in \mathfrak m$, one immediately gets
\begin{equation}
\scal = - \frac{1}{4} \sum_{i,j=1}^{2n} |[e_i,e_j]_{\mathfrak m}|^2 - \frac{1}{2} \sum_{i=1}^{2n} \tr(\ad_{e_i}^2) + \tr(\ad_{H\xi}),
\end{equation}
whence the statement follows by formula \eqref{eq::s=scal+2snormN} and corollary \ref{cor::Hermitianscalarcurvhom}.
%\begin{eqnarray*}
%|N_H|^2
%&=& - \frac{1}{2}\scal + \frac{1}{2} s \\
%&=& \frac{1}{8} \sum_{i,j=1}^{2n} |[e_i,e_j]_{\mathfrak m}|^2 + \frac{1}{4} \sum_{i=1}^{2n} \tr(\ad_{e_i}^2) - \frac{1}{2} \tr(\ad_{H\xi}) \\
%&& + \frac{1}{2} \tr(\ad_{H\xi}-H\ad_{\xi}) \\
%&=& \frac{1}{8} \sum_{i,j=1}^{2n} |[e_i,e_j]_{\mathfrak m}|^2 + \frac{1}{4} \sum_{i=1}^{2n} \tr(\ad_{e_i}^2) - \frac{1}{2} \tr(H\ad_\xi)
%\end{eqnarray*}
\end{proof}

\subsection{Almost K\"ahler structures on coset spaces}

Up to now we considered a manifold endowed with a homogeneous almost-K\"ahler metric and we investigated some of its curvature properties by looking at the relevant Lie algebra.
Conversely, given a connected Lie group $G$ and an even-codimensional compact subgroup $K\subset G$, a homogeneous almost K\"ahler structure on the coset space $M=G/K$ is completely determined by a suitable linear structure on the Lie algebra $\mathfrak g$.
The remainder of this section is devoted in showing this.
In order to have $G$ acting almost effectively on $M$, we also assume that $K$ contains no non-discrete normal subgroups of $G$.

Let $\mathfrak g$ be the Lie algebra of $G$, and let $\mathfrak k \subset \mathfrak g$ be the Lie algebra of $K$.
By our assumptions $\mathfrak k$ has even codimension, say $2n$, inside $\mathfrak g$.
Therefore we can consider the adjoint action of $K$ on $\mathfrak g$ and, thanks to compactness of $K$, choose a $K$-invariant $2n$-dimensional subspace $\mathfrak m \subset \mathfrak g$ such that
\begin{equation}
\mathfrak g = \mathfrak k \oplus \mathfrak m.
\end{equation}
Let $\sigma$ be a $K$-invariant symplectic form on $\mathfrak m$.
By definition, this means that $\sigma$ is a linear two-form on $\mathfrak m$ such that $\sigma^n$ is non-zero and for all $X,Y, Z \in \mathfrak m$ it holds
\begin{equation}\label{eq:closednessonm}
\sigma([X,Y]_{\mathfrak m},Z) + \sigma([Y,Z]_{\mathfrak m},X) + \sigma([Z,X]_{\mathfrak m},Y) = 0.
\end{equation}
Moreover, $K$-invariance means that 
\begin{equation}
\sigma(\Ad_k(X),\Ad_k(Y)) = \sigma(X,Y)
\end{equation}
for all $X,Y \in \mathfrak m$ and for all $k \in K$.

In particular, $(\mathfrak m, \sigma)$ is a symplectic vector space endowed with a (symplectic) representation of $K$. 
Let $H$ be a $K$-invariant compatible almost complex structure on $\mathfrak m$, that is an endomorphism of $\mathfrak m$ commuting with $\Ad_g$ for all $g \in K$, such that $-H^2$ is the identity on $\mathfrak m$ and satisfying the following compatibility conditions
\begin{equation}
\sigma(HX,HY)=\sigma(X,Y), \qquad \sigma(X,HX)>0,
\end{equation}
for all non-zero $X,Y \in \mathfrak m$.
Showing the existence of such a $K$-invariant compatible almost complex structure $H$ is not completely trivial, whereas it can be done by standard methods.
For convenience of the reader we include here a proof of the following
\begin{prop}\label{prop::existenceH}
The space of $K$-invariant almost-complex structures on $\mathfrak m$ compatible with $\sigma$ is non-empty.
\end{prop}
\begin{proof}
Let us start with any $K$-invariant scalar product $h$ on $\mathfrak m$, whose existence is guaranteed by compactness of $K$ and standard average argument.
Whereas $h$ may be completely unrelated to the symplectic form $\sigma$, non degeneracy of $h$ implies there exists a unique $A \in \End(\mathfrak m)$ such that $h(AX,Y)=\sigma(X,Y)$. 
Clearly $A$ is skew-symmetric with respect to $h$. 
Moreover, since both $h$ and $\sigma$ are $K$-invariant, $A$ is $K$-invariant as well.
Note that $-A^2$ is symmetric and positive definite.
Hence there exists a unique symmetric and positive definite $B \in \End(\mathfrak m)$ such that $B^2=-A^2$ and $J=B^{-1}A$.
In order to check that $A$ and $B$ commute, let $\mathfrak m_j$ be the eigenspace of $B$ with eigenvalue $\lambda_j$.
Since $\mathfrak m_j$ is also an eigenspace for $-A^2$ with eigenvalue $\lambda_j^2$, for all $X \in \mathfrak m_j$ one has
\begin{equation}
(-A^2)AX = \lambda_j^2 AX,
\end{equation}
whence $AX \in \mathfrak m_j$, and $A, B$ commute as stated.
Similarly one shows that $B$ is $K$-invariant, for one has
\begin{equation}
B^2\Ad_kX = -\Ad_k A^2 X = \lambda_j^2 \Ad_kX,
\end{equation}
whence $\Ad_kX \in \mathfrak m_j$ for all $k \in K$.
Finally it is easy to check that $H=B^{-1}A$ is a $K$-invariant almost complex structure on $\mathfrak m$ compatible with $\sigma$.
\end{proof}

Now let $M$ be the left-coset space $G/K$.
Let $x \in M$ be the class of the identity element of $G$, and consider the tangent space $T_xM$.
The linear map from $\mathfrak m$ to $T_xM$ which transforms $X$ into $\left.\frac{d}{dt}\right|_{t=0}[\exp(tX)]$ is an isomorphism that intertwines the adjoint action of $K$ on $\mathfrak m$ with the isotropy action of $K$ in $T_xM$.
Moreover, the two form $\sigma$ and the endomorphism $H$ introduced above on $\mathfrak g$ define a $G$-invariant symplectic form $\omega$ and a $G$-invariant compatible complex structure $J$ on $M$.
More specifically, denoted by $l_g$ the left action of $g \in G$ on $M$, they are defined by
\begin{equation}\label{eq::defhomogeneousomegaandJ}
\omega(u_1,u_2) = \sigma(X_1,X_2), \qquad
Ju_1 = d{l_g}(HX_1)
\end{equation}
for all $u_i \in T_{[g]}M$, and $X_i \in \mathfrak m$ such that $dl_{g^{-1}} u_i = X_i$.
Thanks to the hypotheses on $\sigma$ and $H$, it follows that $\omega$ and $J$ are well defined.
In particular, replacing $g$ with $gk$ for some $k \in K$ gives $dl_{(gk)^{-1}}u_i = dl_{k^{-1}}X_i = \Ad_{k^{-1}}(X_i)$.
Therefore the right hand sides of \eqref{eq::defhomogeneousomegaandJ} do not change for $\sigma$ is $K$-invariant and $H$ commutes with the adjoint action of $K$.
Finally, note that closedness of $\omega$ follows by \eqref{eq:closednessonm}.
Thus we have proved the following

\begin{thm}\label{thm::constructionhomogeneousalmostkahler}
Let $G$ be a connected Lie group, and let $K \subset G$ be an even-dimensional compact subgroup which contains no non-discrete normal subgroups of $G$. 
Let $M$ be the coset space $G/K$, and denote by $\mathfrak g$ and $\mathfrak k$ the Lie algebras of $G$ and $K$ respectively. 
Fix a $K$-invariant subspace $\mathfrak m \subset \mathfrak g$ such that $\mathfrak g = \mathfrak k \oplus \mathfrak m$.
Then, given a $K$-invariant symplectic form $\sigma$ on $\mathfrak m$ satisfying
\begin{equation}\label{eq:closedonm}
\sigma([X,Y]_{\mathfrak m},Z) + \sigma([Y,Z]_{\mathfrak m},X) + \sigma([Z,X]_{\mathfrak m},Y) = 0,
\end{equation}
for all $X,Y,Z \in \mathfrak m$, and a $K$-invariant compatible complex structure $H$ on $\mathfrak m$, 
it is defined a homogeneous almost-K\"ahler structure on $M$ by letting
\begin{equation}
\omega(u,v) = \sigma(dl_{g^{-1}}u,dl_{g^{-1}}v), \qquad
Ju = dl_g H dl_{g^{-1}} u
\end{equation}
for all $u,v \in T_{[g]}M$.
\end{thm}

This result then provides an effective way for producing homogeneous almost-K\"ahler structures on $G/K$ just by studying suitable linear structures on $\mathfrak m$.
As we already seen with the complex structure $H$, sometimes is useful to extend those structures to the whole $\mathfrak g$. Therefore, let us extend $\sigma$ to a two-form on $\mathfrak g$ by letting $\sigma(X,Y) = 0$ for all $X \in \mathfrak k$ and $Y \in \mathfrak g$.
Such a $\sigma$ stay $K$-invariant and satisfies 
\begin{itemize}
\item $\sigma([X,Y]_{\mathfrak g},Z) + \sigma([Y,Z]_{\mathfrak g},X) + \sigma([Z,X]_{\mathfrak g},Y) = 0$ for all $X,Y,Z \in \mathfrak g$,
\item $\sigma(X,Y) = 0$ for all $X \in \mathfrak k$, $Y \in \mathfrak g$,
\item the $n$-th wedge product $\sigma^n$ is nonzero.
\end{itemize}
On the other hand, it is easy to see that any $K$-invariant two-form on $\mathfrak g$ satisfying the three conditions above restricts to a $K$-invariant symplectic form $\sigma$ on $\mathfrak m$ satisfying \eqref{eq:closedonm}.
%Recalling that we extended $H$ to $\mathfrak g$ by asking that $H\mathfrak k = \{0\}$
%\begin{itemize}
%\item $H\mathfrak k = \{0\}$,
%\item $[H,\Ad_k]=0$ for all $k \in K$,
%\item $-H^2$ is the projection to $\mathfrak m$,
%\item $\sigma(HX,HY) = \sigma(X,Y)$ for all $X,Y \in \mathfrak g$,
%\item $\sigma(X,HX)>0$ for all non-zero $X \in \mathfrak m$.
%\end{itemize}
%Note that $H$ restricts to a complex structure on $\mathfrak m$ which is compatible with the symplectic form induced by $\sigma$.
%Conversely, any such complex structure extends uniquely to an endomorphism of $\mathfrak g$ satisfying conditions above.

Note that the first condition means that $\sigma$ is a closed two-cocycle of the Chevalley-Eilenberg complex of $\mathfrak g$ and, as a consequence, it represents a class $[\sigma]$ inside the Lie algebra cohomology group $H^2(\mathfrak g)$.
Full details of that theory can be find in the work of Chevalley and Eilenberg \cite{ChevallayEilenberg1947}.
Here we recall just few elementary facts. A one-cochain of that theory is nothing but a linear one-form $\theta \in \mathfrak g^*$, and its differential $\delta\theta$ is the two-form defined by
\begin{equation}
\delta\theta(X,Y) = - \theta([X,Y]_{\mathfrak g}).
\end{equation}
Note that by corollary \ref{cor::rhohomwithW} the Chern-Ricci form of a homogeneous compatible almost complex structure on a symplectic manifold is determined by the exact two-form $-\delta\zeta$.
More specifically, similarly to \eqref{eq::defhomogeneousomegaandJ} one has
\begin{equation}
\rho(u_1,u_2) = \zeta([X_1,X_2]_{\mathfrak g})
\end{equation}
for all $u_i \in T_{[g]}M$, and $X_i \in \mathfrak m$ such that $dl_{g^{-1}} u_i = X_i$.
Formula above yields a canonical representative of $4\pi c_1 \in H_{dR}^2(M)$ which in general is simpler than $\rho$.
To see this consider the decomposition $\zeta = \zeta_{\mathfrak k} + \zeta_{\mathfrak m}$ according to $\mathfrak g = \mathfrak k \oplus \mathfrak m$, so that equation above can be rewritten as
\begin{equation}\label{eq::decompositionchernricciform}
\rho(u_1,u_2) = \zeta_{\mathfrak k}([X_1,X_2]_{\mathfrak k}) + \zeta_{\mathfrak m}([X_1,X_2]_{\mathfrak m}).
\end{equation}
Now observe that by the relation between Lie brackets given in \eqref{eq::relationbetweenbrackets} it follows that the second summand above is equal to $d\alpha(u_1,u_2)$, where $\alpha$ is the $G$-invariant differential one-form on $M$ defined by $\alpha(u_1) = -\zeta_{\mathfrak m}(X_1)$.
Moreover we highlighted after corollary \ref{cor::rhohomwithW} that $\zeta_{\mathfrak k}$ is determined by its restriction to the center $\mathfrak z$ of $\mathfrak k$. Hence one has the following
\begin{prop}
The first Chern class of $M$ is represented by the $G$-invariant differential two-form $\rho'$ on $M$ defined by
\begin{equation}
\rho'(u_1,u_2) = \zeta([X_1,X_2]_{\mathfrak z}),
\end{equation}
for all $u_i \in T_{[g]}M$, and $X_i \in \mathfrak m$ such that $dl_{g^{-1}} u_i = X_i$.
\end{prop}
A straightforward consequence is the following
\begin{cor}\label{cor::Kdiscreteimpliesc_1=0}
If $K$ has discrete center, then $(M,\omega)$ has $c_1=0$.
\end{cor}
On the other hand, assuming that $(\mathfrak k, \mathfrak m)$ is a Cartan pair of $\mathfrak g$, forces $J$ to be integrable.
This follows directly by lemma \ref{lem::Nijenhuishomsymm} after realizing that in our situation $(\mathfrak k, \mathfrak m)$ is a Cartan pair if and only if $[\mathfrak m, \mathfrak m] \subset \mathfrak k$.
Therefore it holds the following
\begin{cor}
If $\mathfrak g = \mathfrak k \oplus \mathfrak m$ with centerless $\mathfrak k$ and $[\mathfrak m, \mathfrak m] \subset \mathfrak k$, then any homogeneous compatible almost complex structure on $(M,\omega)$ is integrable and satisfies $\rho = 0$.
\end{cor}

%
%Given $\sigma$ and $H$, one defines a symmetric bilinear form on $\mathfrak g$ by letting
%\begin{equation}
%\langle X,Y \rangle = \sigma (X,HY).
%\end{equation}
%It is easy to check that such bilinear form restricts to a positive definite scalar product on $\mathfrak m$ and satisfies
%\begin{itemize}
%\item $\langle [Z,X],Y \rangle + \langle X,[Z,Y] \rangle = 0$ for all $X,Y \in \mathfrak m$, $Z \in \mathfrak k$,
%\item $\langle X,Z \rangle = 0$ for all $X \in \mathfrak g$, $Z \in \mathfrak k$,
%\item $\langle HX,HY \rangle = \langle X,Y \rangle$ for all $X,Y \in \mathfrak g$.
%\end{itemize}

\section{Applications}\label{sec::applications}

In this section we apply the theory developed in the previous section in order to produce examples of symplectic manifolds endowed with (homogeneous) special compatible almost complex structures.

\subsection{Symplectic Lie groups}\label{subsec::sympliegroups}

A symplectic Lie group is a connected Lie group $G$ equipped with a left-invariant symplectic structure $\omega$.
Hence it fits in the picture discussed in the previous section by taking $M=G$ and $K=\{e\}$.
In particular, the symplectic form $\omega$ is determined by a linear symplectic form $\sigma$ on $\mathfrak g$ satisfying 
\begin{equation}\label{eq::closednessforgroups}
\sigma([X,Y]_{\mathfrak g},Z) + \sigma([Y,Z]_{\mathfrak g},X) + \sigma([Z,X]_{\mathfrak g},Y) = 0,
\end{equation}
for all $X,Y,Z \in \mathfrak g$.
Note that by corollary \ref{cor::Kdiscreteimpliesc_1=0}, any symplectic Lie group has vanishing first Chern class.

Symplectic Lie group satisfy quite stringent algebraic conditions.
Indeed, Chu proved that a unimodular symplectic group has to be solvable, and in case $G$ has dimension four the same is true even dropping the unimodularity assumption \cite[Theorems 9 and 11]{Chu1974}.
This rules out two important classes of Lie groups.
In fact, mentioned results of Chu imply that semi-simple Lie groups cannot be symplectic and any compact symplectic Lie group must be a torus.

Any left-invariant compatible almost complex structure on $(G,\omega)$ is determined by a linear complex structure on $\mathfrak g$ compatible with $\sigma$.
Therefore, if $2n$ is the dimension of $G$, then fixing a reference symplectic basis of $\mathfrak g$ and a reference compatible almost complex structure has the effect of identifying $\mathfrak g$ with $\mathbf R^{2n}$ equipped with the standard symplectic structure $\omega_0$ and the standard complex structure $J_0$.
Consequently, the space of left-invariant compatible almost complex structures on $(G,\omega)$ is identified with the Siegel upper half-space $Sp(2n,\mathbf R) / U(n)$.
Here we think of $Sp(2n,\mathbf R)$ as the group of all real $2n\times 2n$-matrices $g$ such that $\omega_0(gu,gv)=\omega_0(u,v)$ for all $u,v \in \mathbf R^{2n}$, and $U(n) \subset Sp(2n,\mathbf R)$ as the subgroup constituted by all $g$'s that commute with $J_0$.
The Siegel half-space is a co-adjoint orbit for the group $Sp(2n,\mathbf R)$ and, as such, it admits a (homogeneous) symplectic structure and a Hamiltonian action of $Sp(2n,\mathbf R)$ with moment map $\mu : Sp(2n,\mathbf R) / U(n) \to \mathfrak {sp}(2n,\mathbf R)^*$ given by
\begin{equation}\label{eq::momentSiegeluhs}
\mu(H)(T) = -\tr(HT) 
\end{equation}
for all $H = gJ_0g^{-1} \in Sp(2n,\mathbf R) / U(n)$, and $T \in \mathfrak {sp}(2n,\mathbf R)$.
This structure is well known, and it fits readily in general theory described in section \ref{sec::homsympman} by considering on $\mathfrak {sp}(2n,\mathbf R)$ the two-form defined by $\tr(J_0[S,T])$, and the complex structure which transforms $T$ to $J_0T$.

By corollary \ref{cor::rhohomwithW} it turns out that, in case $G$ is unimodular, the right hand side of formula \eqref{eq::momentSiegeluhs} gives the Chern-Ricci form of the compatible almost complex structure associated to $H$ once $T$ is replaced with $\ad_{[X,Y]}$ for all $X,Y \in \mathfrak g$.
As a consequence, if the derived sub-algebra $\mathfrak g' = [\mathfrak g, \mathfrak g]$ generated a closed subgroup $G'$ of $G$ and if $G'$ acted by symplectic transformations on $\mathfrak g$, the problem of finding which compatible complex structures on $\mathfrak g$ correspond to left-invariant special compatible almost complex structures on $G$ would fit into a moment map picture.
Unfortunately, in general it seems not to be the case since $G'$ is not closed nor it acts symplectically on $\mathfrak g$.

\subsubsection{The universal cover of the Kodaira-Thurston manifold}\label{subsubsect::univcoverKT}

Consider the four-dimensional real Lie algebra $\mathfrak g$ spanned by $e_1,\dots,e_4$ with commuting relations $[e_1,e_2]=e_4$.
One can check that the simply connected Lie group $G$ associated with $\mathfrak g$ is $\mathbf R^4$ endowed with the product defined by
\begin{equation}
x \cdot y = x+y + x_1y_2e_4
\end{equation}
for all $x,y \in G$.
%\begin{eqnarray*}
%x^{-1}
%&=& -x + x_1x_2e_4 \\
%x.y.x^{-1}
%&=& (x+y + x_1y_2e_4) \cdot (-x + x_1x_2e_4) \\
%&=& x+y + x_1y_2e_4 -x + x_1x_2e_4 + (x_1+y_1)(-x_2)e_4 \\
%&=& y + (x_1y_2 - y_1x_2)e_4 \\
%\end{eqnarray*}
Perhaps more commonly, $G$ turns out to be isomorphic to the direct product of $\mathbf R$ together with the three-dimensional Heisenberg group, and it can be represented as the group of matrices of the form
\begin{equation}
\left(
\begin{array}{cccc}
e^{x_3} & 0 & 0 & 0 \\
0 & 1 & x_1 & x_4 \\
0 & 0 & 1 & x_2 \\
0 & 0 & 0 & 1
\end{array}
\right).
\end{equation}
However, we won't need a specific representation of $G$.
In fact, according to theory developed in section \ref{sec::homsympman} we will equip $G$ with a homogeneous symplectic structure and a compatible special almost complex structure just by working on $\mathfrak g$.
In order to do this, consider the dual basis $\varphi^1,\dots,\varphi^4 \in \mathfrak g^*$ of the basis $e_1,\dots,e_4$ and let
\begin{equation}
\sigma = \varphi^1\wedge \varphi^3 + \varphi^2 \wedge \varphi^4
\end{equation}
and
\begin{equation}
H = e_3 \otimes \varphi^1 - e_1 \otimes \varphi^3 + e_4 \otimes \varphi ^2 - e_2 \otimes \varphi^4.
\end{equation}
Therefore $\sigma$ and $H$ correspond to the standard symplectic and complex structures on $\mathbf R^4$ once this is identified with $\mathfrak g$ via the basis $e_1,\dots,e_4$.
One can readily check that $\sigma$ is closed, i.e. it satisfies \eqref{eq::closednessforgroups}. 
Moreover, note that the basis $e_1,\dots,e_4$ turns out to be both symplectic and orthonormal with respect the scalar product induced by $\sigma$ and $H$, and it satisfies $He_1=e_3$, $He_2=e_4$. As a consequence, one calculates $\xi = [e_1,e_3] + [e_2,e_4] = 0$ whence, in particular by proposition \ref{prop::propertiesxi}, $G$ turns out to be unimodular.

As above, let $\omega$ and $J$ be the left-invariant symplectic and complex structure induced on $G$ by $\sigma$ and $H$ respectively.
By corollary \ref{cor::rhohomwithW} and discussion right after that, the Chern-Ricci form $\rho$ of $J$ is completely determined by the one-form $\zeta \in \mathfrak g^*$ defined by \eqref{eq::defzeta}.
In the current sitation $\zeta$ is given by
\begin{equation}
\zeta(X) 
%= - \tr(H\ad_X)
= \sum_{i=1}^4 -\varphi^i(H[X,e_i])
%= \sum_{i,j=1}^4 -\varphi^j(X)\varphi^i(H[e_j,e_i])
= -\varphi^1(X)\varphi^2(He_4) + \varphi^2(X)\varphi^1(He_4)
%= \varphi^1(X),
\end{equation}
whence $\zeta = \varphi^1$.
Therefore, by corollary \ref{cor::rhohomwithW} and commuting relations of $\mathfrak g$ it follows that 
\begin{equation}
\rho=0.
\end{equation}
As a consequence, the Hermitian scalar survature $s$ of $J$ vanishes.
Finally we calculate the norm of the Nijenhuis tensor of $J$.
Recalling that the only non-zero commuting relation is $[e_1,e_2]=e_4$, it follows readily that $\ad_{e_i}^2=0$ for all $i=1,\dots,4$.
Moreover we already noticed that $\xi=0$.
Therefore, proposition \ref{prop::formulaNijenhuis} yields
\begin{equation}
|N|^2 =\frac{1}{4}.
\end{equation}
Substituting in \eqref{eq::s=scal+2snormN} then the scalar curvature of the homogeneous Riemannian metric associated to $J$ turns out to be $\scal = -1/2$.

\subsubsection{Symplectic two-step nilpotent Lie groups}\label{subsubsection::Symp2stnil}

The example presented above fits in the class of two-step nilpotent Lie group, and shares similar properties with any symplectic Lie group of that class.
Examples of this subsection have been previously investigated by Vezzoni in a more general framework \cite{Vezzoni2013}.

Let $(G,\omega)$ be a symplectic two-step nilpotent Lie group of dimension $2n$.
By definition, $G$ is non-abelian and one has $[[X,Y],Z]=0$ for all $X,Y,Z \in \mathfrak g$.
Any compatible almost complex structure $J$ on $(G,\omega)$ is special.
To see this, note first of all that $G$ is unimodular in that it is nilpotent.
Moreover, two-step nilpotency assumption implies that $\ad_{[X,Y]}=0$ for all $X,Y \in \mathfrak g$.
Therefore corollary \ref{cor::rhohomwithW} yields $\rho=0$, independently of the chosen compatible complex structure $H$ on $\mathfrak g$.
Similarly, two-step nilpotency simplifies the formula for the squared norm of the Nijenhuis tensor $N$ of $J$ given in proposition \ref{prop::formulaNijenhuis}.
In fact, both $\ad_{e_i}^2$ and $\ad_\xi$ vanish, whence
\begin{equation}
|N|^2 =\frac{1}{8} \sum_{i,j=1}^{2n} |[e_i,e_j]|^2,
\end{equation}
being ${e_i}$ a symplectic basis of $\mathfrak g$ such that $e_{i+n}=He_i$.
Therefore, none of the left-invariant compatible almost complex structures on $(G,\omega)$ is integrable and, in general, $|N|^2$ turns out to depend effectively on $H$.

\subsection{Co-adjoint orbits}\label{subsect::coadjointorbits}

Given a connected Lie group $G$, consider the co-adjoint action of $G$ on $\mathfrak g^*$ defined by
\begin{equation}
\Ad_g^*\theta = \theta \circ \Ad_g^{-1}
\end{equation}
for all $\theta \in \mathfrak g^*$.
Let $M$ be a co-adjoint orbit, that is the orbit of an element $\theta$ under this action. 
A key observation due to Kirillov, Kostant and Souriau is that any co-adjoint orbit carries a canonical symplectic form making it a homogeneous symplectic manifold.
In order to fit into the picture of section \ref{sec::homsympman}, henceforth we assume that $\theta \in \mathfrak g^*$ has isotropy group $K \subset G$ which contains no non-discrete normal subgroups of $G$.
In particular, $G$ is forced to have discrete center.
Thus the canonical symplectic form $\omega$ on the orbit $M=G/K$ is determined by the exact two-cocycle $\sigma$ of $\mathfrak g$ defined by
\begin{equation}\label{eq::defsigmacoadjointorbit}
\sigma(X,Y) = \theta([X,Y]_{\mathfrak g}).
\end{equation}

Co-adjoint orbits are a class of symplectic manifolds deeply investigated by many viewpoints ranging from quantum mechanics to representation theory \cite{Vogan2000}.
Their importance in the context of homogeneous special compatible almost complex structures stand on the remarkable fact that such structures having non-zero Hermitian scalar curvature may exist only on covering spaces of co-adjoint orbits. More precisely we have the following
\begin{thm}
Let $(M,\omega)$ be a symplectic manifold admitting a homogeneous compatible almost complex structure satisfying $\rho = \lambda \omega$ for some $\lambda \neq 0$.
Then $M$ is a covering space of a co-adjoint orbit and $\omega$ is the pull-back via the covering map of the canonical symplectic form.
\end{thm}
\begin{proof}
Let $J$ be a compatible almost complex structure on $(M,\omega)$ as in the statement.
Keep notation of section \ref{sec::homsympman}.
By evaluating at $x$ the equation $\rho=\lambda\omega$ and substituting formula \eqref{eq::rhoatx} yileds the identity
\begin{equation}\label{eq::lambdasigma=-diffzeta}
\zeta([X,Y]_{\mathfrak g}) = \lambda \sigma (X,Y).
\end{equation}
Let $\theta = \lambda^{-1} \zeta \in \mathfrak g^*$, and let $K' \subset G$ be the stabilizer of $\theta$ under the co-adjoint action of $G$.
Note that $K \subset K'$ for $\zeta$ is $K$-invariant.
On the other hand, $K'$ must have the same dimension of $K$.
If not, let $Y \in \mathfrak k' \setminus \mathfrak k$ be non-zero, and write $Y=Y_{\mathfrak m} + Y_{\mathfrak k}$ with $Y_{\mathfrak m} \neq 0$, so that $\sigma(X,Y_{\mathfrak m}) = 0$ for all $X \in \mathfrak m$, contradicting the fact that $\sigma$ is non-degenerate on $\mathfrak m$.
The upshot is that $K'/K$ is discrete, hence the obvious map from $G/K$ to $G/K'$ is a covering map from $M$ to the co-adjoint orbit of $\theta$.
Finally note that \eqref{eq::lambdasigma=-diffzeta} means that $\sigma$ equals the (Chevallay-Eilenberg) differential of $-\theta$, whence it is clear that the considered covering map pulls-back the canonical symplectic form of $G/K'$ to $\omega$. 
\end{proof}

It would be of some interest understanding if a sort of converse of theorem above holds true.
More precisely, one may expect that a co-adjoint orbit having the first Chern class equal to a non-zero multiple of the class of the canonical symplectic form admits a homogeneous special compatible almost complex structure.
The remainder of this section is devoted to show that this is the case under the additional assumptions that $G$ is semi-simple and the co-adjoint orbit has compact isotropy.
This will be enough to produce examples of co-adjoint orbits equipped with homogeneous special compatible almost complex structures with no restriction on the sign of the Hermitian scalar curvature, including the case $\lambda=0$. 
%
%From now on let $(M,\omega)$ be a co-adjoint orbit as introduced at the beginning of this sub-section, and assume that $\mathfrak g$ is a quadratic Lie algebra.
%This means, by definition, that $\mathfrak g$ admits a $G$-invariant non-degenerate symmetric bilinear form $B$. 
%This is equivalent to requiring that $G$ admits a left-invariant pseudo-Riemannian metric. 
%The existence of such a $B$ on a given Lie algebra is a delicate problem \cite{MedinaRevoy1985,FavreSantharoubane1987}. 
%On the other hand, the class of quadratic Lie algebras include semi-simple and compact ones. The former by Cartan criterion of semi-simplicity, and the latter by standard average argument.

From now on let $(M,\omega)$ be a co-adjoint orbit as introduced at the beginning of this sub-section, and assume that $G$ is semi-simple and the isotropy $K$ is compact.
Semi-simplicity means that the Killing-Cartan form of $\mathfrak g$, that is the $G$-invariant symmetric bilinear form $B$ defined by
\begin{equation}
B(X,Y) = \tr(\ad_X \ad_Y),
\end{equation}
is non-degenerate.
A consequence of our assumption is that associated to $\theta$ there is a unique $V \in \mathfrak g$ such that $\theta(X) = B(V,X)$.
Therefore $K$ is the isotropy subgroup of $V$ with respect to the adjoint representation of $G$, and $\mathfrak k$ turns out to be the centralizer of $V$ inside $\mathfrak g$, that is the set of all $X \in \mathfrak g$ such that $[V,X]_{\mathfrak g}=0$.

\begin{lem}\label{lem::Kconnected}
$K$ is connected.
\end{lem}
\begin{proof}
Compactness of $K$ has two important consequences.
First, one can choose a maximal compact subgroup $K' \subset G$ which contains $K$.
Second, the closure $T \subset K$ of the one-parameter subgroup generated by $V$ is a torus.
On the other hand, a consequence of Cartan decomposition is that $G$ retracts to $K'$, whence $K'$ is connected for $G$ is.
At this point note that $K$ is equal to the centralizer of the torus $T$ inside the compact connected Lie group $K'$.
Therefore we can deduce that $K$ is connected by a well known result due to Hopf \cite{Serre1954}.
\end{proof}

Let $\mathfrak m \subset \mathfrak g$ be a $K$-invariant complement of $\mathfrak k$, so that we have the usual splitting $\mathfrak g = \mathfrak k \oplus \mathfrak m$.
Let $2n$ be the dimension of $\mathfrak m$.
We will denote by $\sigma$ both the two-form defined by \eqref{eq::defsigmacoadjointorbit} and its restriction to $\mathfrak m$.
We believe this will be not source of confusion, in that one has $\sigma(X,Y)=0$ for al $X\in \mathfrak k$ and $Y \in \mathfrak g$ thanks to the fact that $K$ is the isotropy subgroup of $\theta$.

By discussion above one has the identity $\sigma(X,Y) = B(V,[X,Y]_{\mathfrak g})$, and by $G$-invariance of the Killing-Cartan form $B$ one can also represent the symplectic form $\sigma$ on $\mathfrak m$ as
\begin{equation}\label{eq::sigmaintermsofB}
\sigma(X,Y) = B([V,X]_{\mathfrak m},Y),
\end{equation}
for all $X,Y \in \mathfrak m$.
Note that $\ad_V$ restricts to an invertible endomorphism of $\mathfrak m$, for $\mathfrak m$ is $K$-invariant and $\mathfrak k$ is the set of all elements of $\mathfrak g$ which commute with $V$.
%This has two direct important consequences.
%First, by proposition \ref{prop::propertiesxi} the fundamental element $\xi \in \mathfrak g$ has trivial component along $\mathfrak m$, hence $G$ has to be unimodular.
A direct consequence of formula \eqref{eq::sigmaintermsofB} is that the restriction to $\mathfrak m$ of $B$ stays non-degenerate.

Now let $T \subset K$ be the torus generated by the center of $\mathfrak k$.
Since $\ad_V$ is invertible on $\mathfrak m$, there are no non-zero $T$-invariant elements in $\mathfrak m$.
Therefore, $\mathfrak m$ decomposes as
\begin{equation}
\mathfrak m = \mathfrak m_1 \oplus \dots \oplus \mathfrak m_n, 
\end{equation}
where $\mathfrak m_i$ is a two-dimensional $T$-invariant symplectic subspace of $\mathfrak m$, on which $B$ restricts to a definite symmetric bilinear form.
Set $\varepsilon_i=1$ if $B$ is positive on $\mathfrak m_i$, and $\varepsilon_i=-1$ otherwise.
Moreover, let $u_i,v_i$ be an orthogonal basis of $\mathfrak m_i$ normalized so that $B(u_i,u_i) = B(v_i,v_i) = \varepsilon_i$.
Since $\ad_V$ is skew-symmetric with respect to $B$, there exists a real $\lambda_i \neq 0$ such that
\begin{equation}
[V,u_i]_{\mathfrak g} = \lambda_i v_i, \qquad [V,v_i]_{\mathfrak g} = - \lambda_i u_i.
\end{equation}
Perhaps switching the role of $u_i$ and $v_i$, one can assume that $\lambda_i>0$.
Now it is immediate to check that $\varepsilon_i/\lambda_i \ad_V$ defines a compatible complex structure on $\mathfrak m_i$.
Hence the upshot is that letting
\begin{equation}\label{eq::defHadjointorbits}
HX = \sum_{i=1}^n \frac{\varepsilon_i}{\lambda_i} [V,X_i]_{\mathfrak g},
\end{equation}
where $X \in \mathfrak m$ and $X_i$ is the component of $X$ along $\mathfrak m_i$, defines a compatible $K$-invariant complex structure on $\mathfrak m$.
Moreover note that defining $e_i = (1/\sqrt{\lambda_i}) u_i$, and $e_{i+n} = (\varepsilon_i/\sqrt{\lambda_i})v_i$ gives a symplectic basis of $\mathfrak m$ such that $e_{i+n}=He_i$.

By theorem \ref{thm::constructionhomogeneousalmostkahler} then $H$ induces on $M$ an almost-complex structure $J$ compatible with $\omega$.
Moreover, by corollary \ref{cor::rhohomwithW}, the Chern-Ricci form $\rho$ of $J$ is determined by the linear one-form $\zeta \in \mathfrak g^*$ defined by $\zeta(X) = \tr(\ad_{HX}-H\ad_X)$.
Non-degereacy and $G$-invariance of the Killing-Cartan form $B$ readily imply that $G$ is unimodular, hence $\tr(\ad_{HX})$ vanishes.
Moreover, thanks to \eqref{eq::defHadjointorbits} and the basis introduced above one has
%\begin{eqnarray*}
%\zeta(X)
%&=& -\tr(H\ad_X) \\
%&=& - \sum_{i=1}^n \sigma(H[X,e_i],He_i) + \sigma(H[X,He_i],H^2e_i) \\
%&=& - \sum_{i=1}^n \sigma([X,e_i],e_i) + \sigma([X,He_i],He_i)
%\end{eqnarray*}
\begin{equation}
\zeta(X) = - \sum_{i=1}^n \sigma([X,e_i]_{\mathfrak m},e_i) + \sigma([X,He_i]_{\mathfrak m},He_i)
\end{equation}
for all $X\in \mathfrak g$.
Substituting \eqref{eq::sigmaintermsofB} and recalling that $B$ is $G$-invariant yields
%\begin{eqnarray*}
%\zeta(X)
%&=& - \sum_{i=1}^n B([V,[X,e_i]],e_i) + B([V[X,He_i]],He_i) \\
%&=& \sum_{i=1}^n B([X,e_i],[V,e_i]) + B([X,He_i],[V,He_i]) \\
%&=& - \sum_{i=1}^n B(X,[[V,e_i],e_i]) + B(X,[[V,He_i],He_i])
%\end{eqnarray*}
\begin{equation}
\zeta(X) = - \sum_{i=1}^n B(X,[[V,e_i]_{\mathfrak m},e_i]_{\mathfrak g}) + B(X,[[V,He_i]_{\mathfrak m},He_i]_{\mathfrak g}),
\end{equation}
whence, by the expression of $e_i,He_i$ in terms of $u_i,v_i$ one finally gets
\begin{equation}
\zeta(X) = B(V',X),
\end{equation}
being $V'= 2\sum_{i=1}^n [u_i,v_i]_{\mathfrak g}$.
Note that $V'$ belongs to the center of $\mathfrak k$, since $\zeta$ is $K$-invariant and $B$ is $G$-invariant and non-degenerate.
%\begin{eqnarray*}
%B(\Ad_k(V'),X)
%&=& B(V',\Ad_k^{-1}(X)) \\
%&=& \zeta(\Ad_k^{-1}(X)) \\
%&=& \zeta(X) \\
%&=& B(V',X)
%\end{eqnarray*}

At this point, note that in case $V' = \lambda V$ for some real $\lambda$, one has $\zeta = \lambda \theta$ which in turn implies $\rho = \lambda \omega$, so that $J$ is a special compatible almost complex structure on $(M,\omega)$.
Conversely, if $J$ satisfies $\rho = \lambda \omega$, then semi-semplicity of $\mathfrak g$ yields $V'=\lambda V$.
To see this, note that evaluating at $x$ the condition $\rho = \lambda \omega$ yields $\zeta([X,Y]_{\mathfrak g}) = \lambda \sigma (X,Y)$ for all $X,Y \in \mathfrak m$, whence
\begin{equation}\label{eq::specialequivalentequation}
B(V'-\lambda V,[X,Y]_{\mathfrak g}) = 0.
\end{equation}
Now, semi-simplicity of $\mathfrak g$ implies $[\mathfrak g,\mathfrak g] = \mathfrak g$, whence it follows readily that any element belonging to the center of $\mathfrak k$ is of the form $[X,Y]_{\mathfrak g}$ with $X,Y \in \mathfrak m$.
In particular, \eqref{eq::specialequivalentequation} implies $B(V'-\lambda V,V'-\lambda V) = 0$.
Since $K$ is compact and $\mathfrak g$ has trivial center, one can show that $B$ is negative definite on $\mathfrak k$, whence $V'=\lambda V$, as claimed.

Finally we show that $V,V'$ are linearly dependent whenever $[\omega],c_1 \in H_{dR}^2(M)$ are.
To this end, let $\mathfrak z$ be the center of $\mathfrak k$ and consider the linear map 
\begin{equation}
\varphi: \mathfrak z \to \Omega^2(M)
\end{equation}
which transforms $Z$ to the $G$-invariant two-form $\varphi(Z)$ on $M$ defined by $\varphi(Z)(X,Y)_x = B(Z,[X,Y]_{\mathfrak g})$ for all $X,Y \in \mathfrak m$.
Note that by Jacobi identity, $B(Z,[X,Y]_{\mathfrak g})$ defines a closed two form on $\mathfrak g$.
Therefore $\varphi(Z)$ is closed, and one associates $Z$ with a two-cohomology class on $M$.
\begin{lem}\label{lem::injcohomology}
The linear map $\mathfrak z \to H_{dR}^2(M)$ induced by $\varphi$ is injective.
\end{lem}
\begin{proof}
Let $\pi$ be the projection from $G$ to $M$, and consider the induced pull-back map $P$ form $H_{dR}^2(M)$ to $H_{dR}^2(G)$.
We will prove the statement by showing that $\varphi$ induces an isomorphism between $\mathfrak z$ and the kernel of $P$.
First of all, note that $P[\varphi(Z)]=0$ for all $Z \in \mathfrak z$.
This follows from the fact that $P[\varphi(Z)] \in H_{dR}^2(G)$ is represented by the left-invariant differential form on $G$ defined by $\beta(Z)(X,Y)_e = B(Z,[X,Y]_{\mathfrak g})$ for all $X,Y \in \mathfrak g$, which is exact.

On the other hand, let $\gamma$ be a closed two-form on $M$ such that $P[\gamma]=0$.
The pull-back $\pi^*\gamma$ to $G$ is equal to $d\lambda$ for some one-form $\lambda \in \Omega^1(G)$.
Let $\lambda' \in \Omega^1(M)$ be the restriction of $\lambda$ to $K$.
Note that pulling back via $\pi$ a differential form on $M$ and restricting to $K$ gives the zero form.
Indeed this is the same as restricting to $\{x\}$, which has dimension zero, and pulling back via $\pi$ to $K$.
As a consequence, $\lambda'$ is closed, for $d\lambda'$ is equal to the restriction to $K$ of the pull-back $\pi^*\gamma$.
Moreover, note that replacing $\lambda$ with a cohomologous one-form on $G$ has the effect to adding $\lambda'$ with a closed one-form on $K$.
Similarly, replacing $\gamma$ with a cohomologous two-form on $M$ does not change the cohomology class of $\lambda'$.
The upshot is that we defined a map form $\ker P$ to $H_{dR}^1(K)$.

Since we assumed $K$ to be compact and we shown by lemma \ref{lem::Kconnected} that $K$ is connected, by standard averaging argument and by the structure of compact Lie groups it follows that taking the restriction to the identity element of an invariant representative of a class $c \in H_{dR}^1(K)$ induces an isomorphism form $H_{dR}^1(K)$ to $\mathfrak z$.
At this point it is an easy exercise to show that such map is a right-inverse for the map in the statement, which henceforth has to be injective.
\end{proof}

Now come back to the problem of showing that $V,V'$ are linearly dependent if $[\omega],c_1$ are.
First of all note that by means of the map $\varphi$ introduced above one has $\rho = \varphi(V')$, and $\omega = \varphi(V)$. 
Now suppose that $4\pi c_1=\lambda [\omega]$ for some real $\lambda$.
Since $\rho$ represents $4\pi c_1$, it follows that
\begin{equation}
[\varphi(V'-\lambda V)] = 0 \in H_{dR}^2(M),
\end{equation}
whence $V'=\lambda V$ by lemma \ref{lem::injcohomology}.
Summarizing we proved the following

\begin{thm}\label{thm::existencespecialsemisimple}
Let $G$ be a connected semi-simple Lie group, and let $M \subset \mathfrak g^*$ be a co-adjoint orbit equipped with the canonical symplectic form $\omega$.
Assume that the isotropy of $M$ is compact and contains no non-discrete normal subgroups of $G$.
If the first Chern class of $(M,\omega)$ satisfies $4\pi c_1=\lambda [\omega]$ for some $\lambda \in \mathbf R$, then there exists a homogeneous special almost complex structure on $(M,\omega)$.
\end{thm}

\subsubsection{The twistor space of the real hyperbolic 2n-space}\label{subsubsect::twistorHspace}

As is well known, the hyperbolic space $H^{2n}$ can be realized as the homogeneous Riemannian space $SO(2n,1)/SO(2n)$.
More specifically, one considers on $\mathbf R^{2n+1}$ the quadratic form
\begin{equation}
q(v) = v_1^2 + \dots + v_{2n}^2-v_{2n+1}^2,
\end{equation}
so that $SO(2n,1) \subset GL(2n+1,\mathbf R)$ is the identity component of the group constituted by all transformations that preserve $q$.
Moreover, $SO(2n) \subset SO(2n,1)$ is the isotropy subgroup of the point $(0,\dots,0,1) \in \mathbf R^{2n+1}$.
The Lie algebra $\mathfrak g$ of $SO(2n,1)$ turns out to be the subalgebra of $\mathfrak {gl}(2n+1,\mathbf R)$ constituted by matrices of the form
\begin{equation}\label{eq::generalXofso0(2n,1)}
X= \left(
\begin{array}{cc}
A & u \\
u^t & 0
\end{array}
\right)
\end{equation}
where $u \in \mathbf R^{2n}$ and $A \in \mathfrak o(2n)$.
Thanks to the general theory of homogeneous Riemannian manifolds, considering the resulting (Cartan) decomposition $\mathfrak g = \mathbf R^{2n} \oplus \mathfrak o(2n)$ and choosing the standard scalar product on $\mathbf R^{2n}$ defines a homogeneous Riemannian metric on $H^{2n}$ which turns out to have constant sectional curvature equal to $-1$.

After noting that the standard orientation of $\mathbf R^{2n}$ induces an orientation on $H^{2n}$, consider the twistor space of $H^{2n}$, that is the space of all orthogonal orientation-preserving complex structures on $T_xH^{2n}$ as $x$ varies in $H^{2n}$. Since the isotropy action of $SO(2n)$ on $T_xH^{2n}$ is the standard one, $SO(2n,1)$ turns out to act transitively on the twistor space of $H^{2n}$ with isotropy $U(n)$.
Therefore the twistor space of $H^{2n}$ is the homogeneous space $SO(2n,1)/U(n)$.  
We will recall below that this space can be realized as a (co)adjoint orbit of $SO(2n,1)$.
As a consequence, it fits in the general theory developed above, and we will show that it admits a compatible almost complex structure whose Chern-Ricci form satisfies $\rho = (2n-4)\omega$.

To see this, start by considering $Y \in \mathfrak g$ which depends on $v \in \mathbf R^{2n}$ and $B \in \mathfrak o(2n)$, similarly to $X$ as in \eqref{eq::generalXofso0(2n,1)}.
By direct calculation one can check that the Killing-Cartan form of $\mathfrak g$ is given by $(4n-2)B$, being
\begin{equation}\label{eq::Killingso0(2n,1)}
B(X,Y)=\langle u,v \rangle+\frac{1}{2}\tr(AB).
\end{equation}
The reason for working with the resclared Killing-Cartan form $B$ is to avoiding some unpleasant normalizing factors. 
Since $B$ is non-degenerate, one deduces by Cartan criterion that $\mathfrak g$ is semi-simple.
Now let $\theta \in \mathfrak g^*$ be the one-form defined by
\begin{equation}
\theta(X) = B(V,X),
\end{equation}
where $V$ is the element of $\mathfrak g$ defined by
\begin{equation}
V= \left(
\begin{array}{cc}
J_0 & 0 \\
0 & 0
\end{array}
\right),
\end{equation}
and $J_0 \in \mathfrak o(2n)$ is the matrix representing the standard complex structure of $\mathbf R^{2n}$.
The isotropy subgroup $K \subset G$ of $\theta$ coincides with the isotropy subgroup of $V$ and it is isomorphic to the unitary group $U(n)$.
In particular the Lie subalgebra $\mathfrak k \subset \mathfrak g$ is constituted by all $X \in \mathfrak g$ such that $[V,X]=0$.
More concretely, $X$ as in \eqref{eq::generalXofso0(2n,1)} belongs to $\mathfrak k$ if and only if $u=0$ and $A\in\mathfrak o(2n)$ satisfies $[A,J_0]=0$.
Therefore, a $K$-invariant complement $\mathfrak m \subset \mathfrak g$ of $\mathfrak k$ is constituted by all $X$ of the form as in \eqref{eq::generalXofso0(2n,1)} with $A\in\mathfrak o(2n)$ satisfying $[A,J_0] = 2AJ_0$.

By general theory, the co-adjoint orbit of $\theta$ is then the coset space $M=SO(2n,1)/U(n)$ and by discussion above it coincides with the twistor space of $H^{2n}$.
The dimension of $M$ is $n(n+1)$, and the canonical symplectic form $\omega$ on $M$ is determined by the linear two-form $\sigma$ on $\mathfrak g$ defined by
%\begin{eqnarray}
%\sigma(X,Y)
%&=& \theta([X,Y]) \\
%&=& B(V,[X,Y]) \\
%&=& \frac{1}{2} \tr(J_0(uv^t-vu^t+[A,B])) \\
%&=& \langle J_0u,v \rangle + \tr(J_0AB).
%\end{eqnarray}
\begin{equation}
\sigma(X,Y)
= \langle J_0u,v \rangle + \tr(J_0AB).
\end{equation}
The remainder of this section is devoted to show that $(M,\omega)$ admits a special compatible almost complex structure $J$.
To this end, consider the splitting 
\begin{equation}\label{eq::splittingm=m'+m''}
\mathfrak m = \mathfrak m^+ \oplus \mathfrak m^-
\end{equation} where, referring to \eqref{eq::generalXofso0(2n,1)}, $\mathfrak m^+$ is the set of $X \in \mathfrak m$ having $A=0$, and $\mathfrak m^-$ is the set of $X \in \mathfrak m$ such that $u=0$.
Hence, it follows directly by \eqref{eq::Killingso0(2n,1)} that $B$ restricts to a positive definite scalar product on $\mathfrak m^+$ and to a negative definite scalar product on $\mathfrak m^-$.
Since $B$ is $G$-invariant, in particular the one-dimensional torus $T \subset K$ generated by $V$ preserves the splitting \eqref{eq::splittingm=m'+m''}.
Moreover, a $B$-orthonormal basis of $\mathfrak m^+$ is given by
\begin{equation}\label{eq::orthonormalbasisX'}
E_i = \left( \begin{array}{cc} 0 & e_i \\ e_i^t & 0 \end{array} \right),
\end{equation}
where $e_1,\dots,e_{2n}$ is the standard basis of $\mathbf R^{2n}$, and $i$ runs from $1$ to $2n$.
Whereas, a $B$-orthogonal basis of $\mathfrak m^-$ is given by 
\begin{equation}\label{eq::orthogonalbasisX''}
\begin{split}
P_{ij} = \frac{1}{\sqrt{2}}\left([E_i,E_j]-[E_{i+n},E_{j+n}]\right), \\
Q_{ij} = \frac{1}{\sqrt{2}}\left([E_i,E_{j+n}]+[E_{i+n},E_j]\right)
\end{split}
\end{equation}
%\begin{eqnarray*}
%B(P_{ij},P_{ij})
%&=& \frac{1}{4} \tr\left((e_ie_j^t-e_je_i^t-e_{i+n}e_{j+n}^t+e_{j+n}e_{i+n}^t)^2 \right) \\
%&=& \frac{1}{4} \tr\left((e_ie_j^t-e_je_i^t)^2+(e_{i+n}e_{j+n}^t-e_{j+n}e_{i+n}^t)^2 \right) \\
%&=& -1
%\end{eqnarray*}
for all $1\leq i<j \leq n$.
Note that the factor $1/\sqrt{2}$ is placed there in order to making $B(P_{ij},P_{ij})$ and $B(Q_{ij},Q_{ij})$ both equal to $-1$.
The action of $\ad_V$ on $\mathfrak m$ can be conveniently described in the bases just introduced.
In fact, it is easy to check that $[V,E_i] = E_{i+n}$ for all $i \leq n$, 
%\begin{eqnarray*}
%\sqrt{2} [V,P_{ij}]
%&=& [V,[E_i,E_j]]-[V,[E_{i+n},E_{j+n}]] \\
%&=& [[V,E_i],E_j] + [E_i,[V,E_j]]
%-[[V,E_{i+n}],E_{j+n}] - [E_{i+n},[V,E_{j+n}]] \\
%&=& [E_{i+n},E_j] + [E_i,E_{j+n}] + [E_i,E_{j+n}] + [E_{i+n},E_j] \\
%&=& [E_{i+n},E_j] + [E_i,E_{j+n}] + [E_i,E_{j+n}] + [E_{i+n},E_j] \\
%&=& 2 \sqrt{2} Q_{ij}
%\end{eqnarray*}
and $[V,P_{ij}] = 2Q_{ij}$ for all $1\leq i<j \leq n$.
The upshot is that letting $u_i = E_i$, $v_i=E_{i+n}$, and $u_{ij} = P_{ij}$, $v_{ij} = Q_{ij}$ places us in the general situation discussed above (apart from the unimportant fact that some basis elements are indexed by two indices instead of one) with $\lambda_i=1$ and $\lambda_{ij} = 2$.
Therefore the compatible complex structure $J$ we are looking for on $M$ is determined by the linear complex structure $H$ on $\mathfrak m$ defined by $ HX = [V,X^+] - \frac{1}{2}[V,X^-]$,
where $X^+$ and $X^-$ are the component of $X$ according to the splitting \eqref{eq::splittingm=m'+m''}.
Referring again to \eqref{eq::generalXofso0(2n,1)}, then $H$ is given by
\begin{equation}
HX = \left( \begin{array}{cc} -\frac{1}{2}[J_0,A] & J_0u \\ (J_0u)^t & 0 \end{array} \right).
\end{equation}

At this point, showing that $J$ is special is an easy calculation.
Indeed, by general theory we know that $\rho=\lambda \omega$ if and only if $\zeta = \lambda \theta$, where $\zeta \in \mathfrak g^*$ is the one-form defined by $\zeta(X)=B(V',X)$, being $V' \in \mathfrak k$ given by 
\begin{equation}
V' = \sum_{i=1}^n 2[E_i,E_{i+n}] + \sum_{1 \leq i < j \leq n} 2[P_{ij},Q_{ij}].
\end{equation}
By elementary calculations one finds $V'=(2n-4)V$
%\begin{eqnarray*}
%\sum_{i=1}^n [E_i,E_{i+n}]
%&=& \sum_{i=1}^n \left(\begin{array}{cc} e_ie_{i+n}^t-e_{i+n}e_i^t & 0 \\ 0 & 0 \end{array}\right) \\
%&=& -V \\
%\sum_{1 \leq i < j \leq n} [P_{ij},Q_{ij}]
%&=& \frac{1}{2} \sum_{1 \leq i < j \leq n} \left[[E_i,E_j]-[E_{i+n},E_{j+n}],[E_i,E_{j+n}]+[E_{i+n},E_j]\right] \\
%&=& \frac{1}{2} \sum_{1 \leq i < j \leq n} \left(\begin{array}{cc} [e_ie_j^t-e_je_i^t-e_{i+n}e_{j+n}^t+e_{j+n}e_{i+n}^t,e_ie_{j+n}^t-e_{j+n}e_i^t+e_{i+n}e_j^t-e_je_{i+n}^t] & 0 \\ 0 & 0\end{array}\right) \\
%&=& \frac{1}{2} \sum_{1 \leq i < j \leq n} \left(\begin{array}{cc} -2e_je_{j+n}^t-2e_ie_{i+n}^t+2e_{j+n}e_j^t+2e_{i+n}e_i^t & 0 \\ 0 & 0 \end{array}\right) \\
%&=& \sum_{1 \leq i < j \leq n} \left(\begin{array}{cc} e_{i+n}e_i^t-e_ie_{i+n}^t+e_{j+n}e_j^t-e_je_{j+n}^t & 0 \\ 0 & 0 \end{array}\right) \\
%&=& \frac{1}{2} \sum_{i,j=1}^n \left(\begin{array}{cc} e_{i+n}e_i^t-e_ie_{i+n}^t+e_{j+n}e_j^t-e_je_{j+n}^t & 0 \\ 0 & 0 \end{array}\right) \\
%&& - \sum_{i=1}^n \left(\begin{array}{cc} e_{i+n}e_i^t-e_ie_{i+n}^t & 0 \\ 0 & 0 \end{array}\right) \\
%&=& (n-1) V
%\end{eqnarray*}
whence it follows that $J$ satisfies 
\begin{equation}
\rho = (2n-4) \omega.
\end{equation}
As a consequence the Hermitian scalar curvature of $J$ is given by $s=n(n+1)(n-2)$, and that the first Chern class of $(M,\omega)$ satisfies $2\pi c_1 = (n-2)[\omega]$.
In particular, $(M,\omega)$ turns out to be symplectic Calabi-Yau when $n=2$ and symplectic Fano for all $n>2$.

Finally, we calculate the squared norm of the Nijenhuis tensor of $J$.
To this end, note that by observation after formula \eqref{eq::defHadjointorbits} it follows that a symplectic and orthonormal basis of $\mathfrak m$ is given by $E_i, \frac{1}{\sqrt{2}}P_{ij}, E_{i+n}, -\frac{1}{\sqrt{2}}Q_{ij}$.
%As a consequence one has
%\begin{equation}
%\xi = \sum_{i=1}^n [E_i,E_{i+n}] - \frac{1}{2} \sum_{1 \leq i<j \leq n} [P_{ij},Q_{ij}],
%\end{equation}
%whence the same calculations performed for $V'$ yield $\xi = -\frac{n+1}{2} V$.
%\begin{eqnarray*}
%\sum_{i=1}^n [E_i,E_{i+n}]
%&=& \sum_{i=1}^n \left(\begin{array}{cc} 0 & 0 \\ 0 & e_ie_{i+n}^t-e_{i+n}e_i^t \end{array}\right) \\
%&=& -V \\
%\sum_{1 \leq i < j \leq n} [P_{ij},Q_{ij}]
%&=& \frac{1}{2} \sum_{1 \leq i < j \leq n} \left[[E_i,E_j]-[E_{i+n},E_{j+n}],[E_i,E_{j+n}]+[E_{i+n},E_j]\right] \\
%&=& \frac{1}{2} \sum_{1 \leq i < j \leq n} \left(\begin{array}{cc} 0 & 0 \\ 0 & [e_ie_j^t-e_je_i^t-e_{i+n}e_{j+n}^t+e_{j+n}e_{i+n}^t,e_ie_{j+n}^t-e_{j+n}e_i^t+e_{i+n}e_j^t-e_je_{i+n}^t]\end{array}\right) \\
%&=& \frac{1}{2} \sum_{1 \leq i < j \leq n} \left(\begin{array}{cc} 0 & 0 \\ 0 & -2e_je_{j+n}^t-2e_ie_{i+n}^t+2e_{j+n}e_j^t+2e_{i+n}e_i^t \end{array}\right) \\
%&=& \sum_{1 \leq i < j \leq n} \left(\begin{array}{cc} 0 & 0 \\ 0 & e_{i+n}e_i^t-e_ie_{i+n}^t+e_{j+n}e_j^t-e_je_{j+n}^t \end{array}\right) \\
%&=& \frac{1}{2} \sum_{i,j=1}^n \left(\begin{array}{cc} 0 & 0 \\ 0 & e_{i+n}e_i^t-e_ie_{i+n}^t+e_{j+n}e_j^t-e_je_{j+n}^t \end{array}\right) \\
%&& - \sum_{i=1}^n \left(\begin{array}{cc} 0 & 0 \\ 0 & e_{i+n}e_i^t-e_ie_{i+n}^t \end{array}\right) \\
%&=& (n-1) V
%\end{eqnarray*}
In view of applying proposition \ref{prop::formulaNijenhuis}, note that the commutators of basis elements having non-zero component along $\mathfrak m$ are the following
\begin{equation}
\begin{split}
[E_i,E_j]_{\mathfrak m} = - [E_{i+n},E_{j+n}]_{\mathfrak m} = \frac{1}{\sqrt{2}} P_{ij}, \\
[E_i,E_{j+n}]_{\mathfrak m} = [E_{i+n},E_j]_{\mathfrak m} = \frac{1}{\sqrt{2}} Q_{ij}, \\
[P_{ij},E_j] = [Q_{ij},E_{j+n}] = \frac{1}{\sqrt{2}} E_i, \\
[P_{ij},E_i] = [Q_{ij},E_{i+n}] = - \frac{1}{\sqrt{2}} E_j \\
[P_{ij},E_{j+n}] = -[Q_{ij},E_j] = -\frac{1}{\sqrt{2}} E_{i+n}, \\
[P_{ij},E_{i+n}] = -[Q_{ij},E_i] = \frac{1}{\sqrt{2}} E_{j+n},
\end{split}
\end{equation}
where $1 \leq i<j \leq n$.
Moreover note that the second summand of the formula for $|N|^2$ can be expressed in term of the Killing-Cartan form $(4n-2)B$, and finally that the third summand of the same formula is given by $s/2$ since $G$ is unimodular.
Putting all together then proposition \ref{prop::formulaNijenhuis} yields
%\begin{eqnarray*}
%|N|^2
%&=& \frac{1}{8} \sum_{i,j=1}^{\dim M} |[e_i,e_j]_{\mathfrak m}|^2 + \frac{1}{4} \sum_{i=1}^{\dim M} \tr(\ad_{e_i}^2) - \frac{1}{2} \tr(H\ad_\xi) \\
%&=& \frac{1}{4} \left( 6\frac{n(n-1)}{2} \right)
%+ \frac{4n-2}{4} \left(2n - \frac{n(n-1)}{2} \right)
%+ \frac{1}{2}n(n+1)(n-2) \\
%&=& \frac{3}{4}n(n-1)
%- \frac{1}{4}(2n-1)n \left(n-5\right)
%+ \frac{1}{2}n(n+1)(n-2) \\
%&=& \frac{3}{4}n(n-1) -\frac{n}{4}(2n^2-11n+5) +\frac{n}{2}(n^2-n-2) \\
%&=& \frac{1}{4}n(3n-3+11n-5-2n-4) \\
%&=& 3n(n-1)
%\end{eqnarray*}
\begin{equation}
|N|^2 = 3n(n-1).
\end{equation}

\subsubsection{Griffiths period domains of weight two}\label{subsubsection::gentypesymp}

The twistor space of the hyperbolic space $H^{2n}$ described in the previous section is an example of symplectic manifold $(M,\omega)$ that carry a special compatible almost complex structure $J$.
In particular, the first Chern class of such space satisfies $2 \pi c_1 = (n-2) [\omega]$.
Therefore $(M,\omega)$ is of general type only in case $n=1$, when $M$ is the hyperbolic plane and $\omega$ is the K\"ahler form of the hyperbolic metric.
In order to produce examples of symplectic manifolds admitting a non-integrable special compatible almost complex structure with negative Hermitian curvature here we consider a slight generalization of the situation considered above. 
The resulting symplectic manifold $(M,\omega)$ will be a homogeneous fiber bundle on the irreducible non-positively curved symmetric space $SO(2p,q)/(SO(2p)\times SO(q))$, being the latter the space of $2p$-dimensional positive subspaces of a $(2p+q)$-dimensional vector space endowed with a quadratic form of signature $(2p,q)$. 
It is worth to note that $M$ turns out to admit a complex structure, which makes it a Griffiths period domain of weight two \cite[Section 8]{Griffiths1970}.

Given integers $p,q>0$, consider the Lie group $SO(2p,q)$ constituted of all endomorphisms of $\mathbf R^{2p+q}$ that preserves a scalar product of signature $(2p,q)$ and are path connected to the identity.
The Lie algebra $\mathfrak g$ of $SO(2p,q)$ turns out to be the subalgebra of $\mathfrak {gl}(2p+q,\mathbf R)$ constituted by matrices of the form
\begin{equation}\label{eq::generalXofso0(2p,q)}
X= \left(
\begin{array}{cc}
A_{2p} & u \\
u^t & A_q
\end{array}
\right)
\end{equation}
where now $u$ is a matrix belonging to $\mathbf R^{2p \times q}$, $A_{2p} \in \mathfrak o(2p)$, and $A_q \in \mathfrak o(q)$.

Similar to $X$ above, consider $Y \in \mathfrak g$ depending on $v \in \mathbf R^{2p \times q}$, $B_{2p} \in \mathfrak o(2p)$, and $B_q \in \mathfrak o(q)$.
The Killing-Cartan form of $\mathfrak g$ turns out to be $(4p+2q-4)B$, where $B$ is the symmetric bilinear form defined by
\begin{equation}\label{eq::Killingso0(2p,q)}
B(X,Y)=\langle u,v \rangle+\frac{1}{2}\tr(A_{2p}B_{2p}) + \frac{1}{2}\tr(A_qB_q).
\end{equation}
As above, we work with the resclared Killing-Cartan form $B$ in order to get simpler formulae. 
Even in this more general case, $B$ is non-degenerate, whence one deduces by Cartan criterion that $\mathfrak g$ is semi-simple.

Let $\theta \in \mathfrak g^*$ be the one-form defined by $\theta(X) = B(V,X)$, where $V \in \mathfrak g$ is defined by
\begin{equation}
V= \left(
\begin{array}{cc}
J_0 & 0 \\
0 & 0
\end{array}
\right),
\end{equation}
and $J_0 \in \mathfrak o(2p)$ is the matrix representing the standard complex structure of $\mathbf R^{2p}$.
The isotropy subgroup $K \subset G$ of $\theta$ is the isotropy subgroup of $V$ and it is isomorphic to the direct product $SO(q) \times U(p)$.
The Lie subalgebra $\mathfrak k \subset \mathfrak g$ is constituted by all $X \in \mathfrak g$ such that $[V,X]=0$, that is, all $X$ as in \eqref{eq::generalXofso0(2p,q)} such that $u=0$, and $A_{2p}\in\mathfrak o(2p)$ satisfies $[A_{2p},J_0]=0$.
Hence, a $K$-invariant complement $\mathfrak m \subset \mathfrak g$ of $\mathfrak k$ is constituted by all $X$ with $A_q=0$, and $A+{2p}\in\mathfrak o(2p)$ satisfying $[A_{2p},J_0] = 2A_{2p}J_0$.

The upshot is that the co-adjoint orbit of $\theta$ is the coset space $M=SO(2p,q)/(U(p) \times SO(2q))$, the dimension of $M$ is $p(p+2q-1)$, and the canonical symplectic form $\omega$ on $M$ is determined by the linear two-form $\sigma$ on $\mathfrak g$ defined by
%\begin{eqnarray*}
%\sigma(X,Y)
%&=& \theta([X,Y]) \\
%&=& B(V,[X,Y]) \\
%&=& \frac{1}{2}\tr(J_0(uv^t-vu^t+[A_{2p},B_{2p}])) \\
%&=& \langle J_0u,v \rangle + \tr(J_0A_{2p}B_{2p}).
%\end{eqnarray*}
\begin{equation}
\sigma(X,Y)
= \langle J_0u,v \rangle + \tr(J_0A_{2p}B_{2p}).
\end{equation}
Similarly to the situation discussed in the previous section, in order to describe a special compatible almost complex structure $J$ on $(M,\omega)$, is is worth to consider the splitting
\begin{equation}\label{eq::splittingm=m'+m''2p+q}
\mathfrak m = \mathfrak m^+ \oplus \mathfrak m^-,
\end{equation} 
so that $B$ is positive on $\mathfrak m^+$ and negative in $\mathfrak m^-$.
More concretely, $\mathfrak m^+$ turns out to be the set of $X \in \mathfrak g$ with $A_{2p}=0$, $A_q=0$, and $\mathfrak m^-$ the set of $X \in \mathfrak g$ such that $u=0$, $A_q=0$, $A_{2p}J_0+J_0A_{2p}=0$.
%\begin{itemize}
%\item Let $e_1,\dots,e_{2p}$ be the standard basis of $\mathbf R^{2p}$
%\item Let $f_1,\dots,f_q$ be the standard basis of $\mathbf R^q$
%\item $E_{ij}=\left(\begin{array}{cc} 0 & f_je_i^t \\ e_if_j^t & 0 \end{array}\right)$ with $i\in\{1,\dots,2p\}, j \in \{1,\dots,q\}$ is a $B$-orthonormal basis of $\mathfrak m^+$
%\item $[E_{ij},E_{hk}] = \left(\begin{array}{cc} \delta_{ih}(f_jf_k^t-f_kf_j^t) & 0 \\ 0 & \delta_{jk}(e_ie_h^t-e_he_i^t) \end{array}\right)$
%\item $P_{ih} = \frac{1}{\sqrt{2}}\left([E_{ik},E_{hk}] - [E_{i+pk},E_{h+pk}]\right)$ and \\
%$Q_{ih} = \frac{1}{\sqrt{2}}\left([E_{ik},E_{h+pk}] + [E_{i+pk},E_{hk}]\right)$ \\
%with $1\leq i < h \leq p$ constitute an orthogonal basis for $\mathfrak m^-$ such that
%$B(P_{ih},P_{ih}) = B(Q_{ih},Q_{ih}) = -1$
%\item $[V,E_{ij}] = E_{i+pj}$ for all $1 \leq i \leq p$
%\item $[V,P_{ih}] = 2Q_{ih}$ for all $1 \leq i<h \leq 2p$
%\end{itemize}

Arguing as in previous sub-section with obvious adaptations yields a complex structure $H$ on $\mathfrak m$ defined by $HX = [V,X^+] - \frac{1}{2}[V,X^-]$, being $X^+$ and $X^-$ the components of $X$ according to the splitting \eqref{eq::splittingm=m'+m''2p+q}.
Writing $X$ as in \eqref{eq::generalXofso0(2p,q)}, then $H$ is given by
\begin{equation}
HX = \left( \begin{array}{cc} -\frac{1}{2}[J_0,A_{2p}] & J_0u \\ (J_0u)^t & 0 \end{array} \right).
\end{equation}
The compatible almost complex structure $J$ induced by $H$ on $M$ is special.
This follows by arguments along the lines of sub-section above.
The upshot is that the Chern-Ricci form $\rho$ of $J$ satisfies
\begin{equation}
\rho = (2p-2q-2) \omega.
\end{equation}
As a consequence the Hermitian scalar curvature of $J$ is given by $s=p(p+2q-1)(p-q-1)$,
and the first Chern class of $(M,\omega)$ satisfies
\begin{equation}
2\pi c_1 = (p-q-1)[\omega].
\end{equation}
Finally, one calculates that 
\begin{equation}
|N|^2 = 3pq(p-1).
\end{equation}

We note that $M$ can be thought of as the twistor space of a tangent sub-bundle $E$ of the non-compact symmetric space $SO(2p,q)/(SO(2p)\times SO(q))$.
More specifically, $E$ is the bundle associated to the representation of $SO(2p) \times SO(q)$ on $\mathbf R^{2p}$ where the factor $SO(2p)$ acts in the standard way and the factor $SO(q)$ acts trivially.

As we have already mentioned above, $M$ admits a homogeneous complex structure $\tilde J$.
In our notation, such complex structure is induced by the linear complex structure $\tilde H$ on $\mathfrak m$ defined by
\begin{equation}
\tilde HX = \left( \begin{array}{cc} \frac{1}{2}[J_0,A_{2p}] & J_0u \\ (J_0u)^t & 0 \end{array} \right).
\end{equation}
It can be readily checked that $\tilde J$ is integrable by direct calculation of its Nijenhuis tensor.
Moreover, $\tilde J$ turns out to be symplectic with respect to $\omega$, and the resulting metric on $M$ is pseudo-K\"ahler with segnature $(2pq,p^2-p)$.

\subsection{Quotients}\label{subsection::quot}

Let $(G,\omega)$ be a symplectic Lie group, and let $J$ be a left-invariant compatible almost complex structure on $G$.
Given a discrete torsion-free subgroup $\Gamma$ of $G$, both $\omega$ and $J$ descend to the quotient $\Gamma \backslash G$.
More generally, if $(M,\omega)$ is a symplectic homogeneous manifold of the form $M=G/K$ with $K$ compact, and $J$ is a homogeneous compatible almost complex structure on $(M,\omega)$, then any discrete torsion-free subgroup $\Gamma$ of $G$ gives rise to a quotient manifold $M'= \Gamma \backslash M$, and both $\omega$ and $J$ descend to a symplectic form $\omega'$ and a compatible almost complex structure $J'$ on $M'$.
Since $M$ covers $M'$, the local properties of $J$ and $J'$ are identical.
In particular, are constant the Hermitian scalar curvature and the norm of the Nijenhuis tensor of $J'$, and both coincide with the ones of $J$.
Moreover, $J'$ is special if and only if $J$ is.

Thanks to this elementary remark, one can produce several compact non-K\"ahler symplectic manifolds equipped with special compatible almost complex structures.

\subsubsection{The Kodaira-Thurston manifold and two-step nilmanifolds}

As in sub-section \ref{subsubsect::univcoverKT}, consider $G=\mathbf R^4$ endowed with the product defined by
\begin{equation}
x \cdot y = x+y + x_1y_2e_4
\end{equation}
for all $x,y \in \mathbf R^4$.
The points with integral coordinates of $G$ form a lattice $\Gamma$, and the quotient $M=\Gamma \backslash G$ is the Kodaira-Thurston manifold \cite{Thurston1976}.
By discussion above, the left-invariant symplectic and complex structures induced on $G$ by $\sigma$ and $H$ defined as in sub-section \ref{subsubsect::univcoverKT} descend to a symplectic structure $\omega$ and a compatible almost complex structure $J$ on $M$ satisfying $\rho=0$ and $|N|^2=1/4$.

More generally, let $G$ be a nilpotent two-step Lie group, and assume that $G$ admits a left-invariant symplectic structure.
As shown in sub-section \ref{subsubsection::Symp2stnil}, any left-invariant compatible almost complex structure has vanishing Chern-Ricci form.
As a consequence, on a symplectic two-step nilmanifold, i.e. the quotient $M=\Gamma \backslash G$ of a symplectic two-step Lie group $G$ by a lattice $\Gamma$, any homogeneous compatible almost complex structure is Chern-Ricci flat, accordingly to a previous result of Vezzoni \cite[Theorem 1]{Vezzoni2013}.

All examples discussed in this subsection admit no compatible integrable almost complex structures for they violate the hard Lefschetz condition, as proved by Benson and Gordon \cite[Theorem A]{BensonGordon1988}.

\subsubsection{Twistor spaces of hyperbolic manifolds and quotients of period domains of weight two}\label{subsec::quotienttwistorsandperioddomains}

A hyperbolic $2n$-fold is a complete Riemannian $2n$-fold of constant sectional curvature $-1$.
Since its universal cover is the hyperbolic space $H^{2n}$, any hyperbolic $2n$-fold is of the form $\Gamma \backslash H^{2n}$ where $\Gamma \subset SO(2n,1)$ is a torsion-free discrete subgroup.
As a consequence, the twistor space of $\Gamma \backslash H^{2n}$ is the quotient $\Gamma \backslash M$ where $M = SO(2n,1) / U(n)$ is the twistor space of $H^{2n}$.
We shown in section \ref{subsubsect::twistorHspace} that $M$ admits a homogeneous symplectic form $\omega$ and a compatible almost complex structure $J$ satisfying $\rho = (2n-4)\omega$.
Therefore, $\omega$ and $J$ descend to the twistor space of $\Gamma \backslash H^{2n}$ and satisfy the same equation on the Chern-Ricci form.
This gives plenty of symplectic Calabi-Yau six-folds (when $n=2$) and symplectic Fano $n(n+1)$-folds (when $n>2$), including compact examples, which admits special compatible almost complex structures.
Clearly the fundamental group of any such manifold is isomorphic to $\Gamma$.

In order to get general type examples one can start with the symplectic manifold $M=SO(2p,q)/(U(p) \times SO(q))$ considered in sub-section \ref{subsubsection::gentypesymp} and choose a torsion-free discrete subgroup $\Gamma \subset SO(2p,q)$.
By a classical result of Borel, such a $\Gamma$ exists even with the additional feature of being co-compact, i.e. giving a compact quotient $\Gamma \backslash M$ \cite{Borel1963}.
Therefore the symplectic structure $\omega$ and the compatible almost complex structure $J$ on $M$ satisfying $\rho = (2p-2q-2)\omega$ descend to $\Gamma \backslash M$, giving a (compact) symplectic manifold of dimension $p(p+2q-1)$, satisfying $2\pi c_1 = (p-q-1) [\omega]$, and admitting a special compatible almost complex structure.
Note that taking $p \leq q$ yields compact general type symplectic manifolds.

Finally we highlight that all compact examples discussed in this subsection are not of the homotopy type of a compact K\"ahler manifold whenever $SO(2n,1)/SO(2n)$ and $SO(2p,q)/(SO(2p) \times SO(q))$ are not Hermitian symmetric, that is when $n,p>1$ and $q \neq 2$.
This has been shown by Carlson and Toledo by means of harmonic mapping theory \cite[Theorems 8.1 and 8.2]{CarlsonToledo1989}.
Therefore, the resulting compact symplectic manifolds $\Gamma \backslash M$ admit no compatible integrable almost complex structures.
On the other hand, we recalled at the end of subsection \ref{subsubsection::gentypesymp} that $M$ admits a homogeneous complex structure $\tilde J$ which is symplectic with respect to $\omega$, and gives rise to a homogeneous pseudo-K\"ahler metric of signature $(2pq,p^2-p)$ on $M$.
Clearly such a structure descends to the quotient $\Gamma \backslash M$.

\end{document}